\documentclass[12pt,a4paper,reqno]{amsart}

\title{On the gauge group of Galois objects}
\usepackage{amsmath, amssymb, amsthm}
\usepackage[english]{babel}
\usepackage{fullpage}
\usepackage{enumerate}
\usepackage[all]{xy}
\usepackage{graphicx}
\usepackage{mathtools}
\usepackage{mathrsfs}
\usepackage{hyperref}
\usepackage{verbatim}

\author{Xiao Han, Giovanni Landi}
\address[]{\textit{Xiao Han}
\newline \indent SISSA,
via Bonomea 265, 34136 Trieste, Italy}
\email{xiao.han@sissa.it}
\address[]{\textit{Giovanni Landi}
\newline \indent
Universit\`a di Trieste,
Via A. Valerio, 12/1, 34127  Trieste, Italy
\newline \indent
Institute for Geometry and Physics (IGAP) Trieste, Italy 
\newline \indent and INFN, Trieste, Italy}
\email{landi@units.it}

\DeclareMathOperator{\Aut}{Aut}
\DeclareMathOperator{\Hom}{Hom}
\newcommand*{\id}{\textup{id}}%identity map
\newcommand*{\diag}{\textup{diag}}    %diagonal

\numberwithin{equation}{section}

\theoremstyle{plain}

\newtheorem{thm}{Theorem}[section]
\newtheorem{lem}[thm]{Lemma}
\newtheorem{prop}[thm]{Proposition}
 \newtheorem{cor}[thm]{Corollary}
\newtheorem{defi}[thm]{Definition}

\theoremstyle{remark}
\newtheorem{exa}[thm]{Example}
\newtheorem{rem}[thm]{Remark}

\numberwithin{equation}{section}

\newcommand{\nn}{\nonumber}
\newcommand{\ot}{\otimes}

\newcommand{\beq}{\begin{equation}}
\newcommand{\eeq}{\end{equation}}

\newcommand{\A}{\mathcal{A}}
\newcommand{\B}{\mathcal{B}}

\newcommand{\cL}{\mathcal{L}}

\newcommand{\C}{\mathcal{C}}

\newcommand{\IC}{\mathbb{C}}
\newcommand{\IZ}{\mathbb{Z}}

\newcommand{\zero}[1]{{#1}{}_{\scriptscriptstyle{(0)}}}
\newcommand{\one}[1]{{#1}{}_{\scriptscriptstyle{(1)}}}
\newcommand{\two}[1]{{#1}{}_{\scriptscriptstyle{(2)}}}
\newcommand{\three}[1]{{#1}{}_{\scriptscriptstyle{(3)}}}
\newcommand{\four}[1]{{#1}{}_{\scriptscriptstyle{(4)}}}
\newcommand{\five}[1]{{#1}{}_{\scriptscriptstyle{(5)}}}
\newcommand{\six}[1]{{#1}{}_{\scriptscriptstyle{(6)}}}

\newcommand{\tuno}[1]{{#1}{}{}^{\scriptscriptstyle{<1>}}}
\newcommand{\tdue}[1]{{#1}{}{}^{\scriptscriptstyle{<2>}}}

\begin{document}

\begin{abstract}
We study the Ehresmann--Schauenburg bialgebroid of a noncommutative principal bundle as a quantization of the classical gauge groupoid of a principal bundle. When the base algebra is in the centre of the total space algebra, the gauge group of the noncommutative principal bundle is isomorphic to the group of bisections of the bialgebroid. In particular we consider Galois objects (non-trivial noncommutative bundles over a point in a sense) for which the bialgebroid is 
a Hopf algebra. For these we give a crossed module structure for the bisections and the automorphisms of the bialgebroid. Examples include Galois objects of group Hopf algebras and of Taft algebras. 
\end{abstract}

\maketitle
\tableofcontents
\parskip = .75 ex

\section{Introduction}

The study of groupoids on one hand and gauge theories on the other hand is important in different area of mathematics and physics. In particular these subjects meet in the notion of the gauge groupoid associated to a principal bundle. In view of the considerable amount of recent work on noncommutative principal bundles it is important to come up with a noncommutative version of groupoids and their relations to noncommutative principal bundles, for which one needs to have a better understanding of bialgebroids. 

In this paper, we will consider the Ehresmann--Schauenburg bialgebroid associated with a noncommutative principal bundle as a quantization of the classical gauge groupoid. 
Classically, bisections of the gauge groupoid are closely related to gauge transformations. 
Here we show that under some conditions on the base space algebra of the noncommutative principal bundle,  
the gauge group of the principal bundle is group isomorphic to the group of bisections of the corresponding 
Ehresmann--Schauenburg bialgebroid.

For a bialgebroid there is a notion of coproduct and counit but in general not of an antipode. An antipode can be defined for the Ehresmann--Schauenburg bialgebroid of a Galois object which (the bialgebroid that is) is then a Hopf algebra. 
Now, with $H$ a Hopf algebra, a $H$-Galois object is a noncommutative principal bundle over a point in a sense: a $H$-Hopf--Galois extension of the ground field $\IC$. In contrast to the classical situation were a bundle over a point is trivial, for the isomorphism classes of noncommutative principal bundles over a point this need not be the case. Notable examples are group Hopf algebras $\IC[G]$, whose corresponding principal bundle are $\IC[G]$-graded algebras and are classified by the cohomology group $H^2(G, \IC^{\times})$, and Taft algebras $T_{N}$; the equivalence classes of  $T_{N}$-Galois objects are in bijective correspondence with the abelian group $\IC$.

Thus the central part of the paper is dedicated to the Ehresmann--Schauenburg bialgebroid of a Galois object and to the study of the corresponding groups of bisections, be they algebra maps from the bialgebroid to the ground field (and thus characters) or more general transformations. These are in bijective correspondence with the group of gauge transformations of the Galois object. We study in particular the case of Galois objects for $H$ a general cocommutative Hopf algebra and in particular a group Hopf algebra. A nice class of examples comes from Galois objects and corresponding Ehresmann--Schauenburg bialgebroids for the Taft algebras $T_{N}$ an example that we work out in full details. 

Automorphisms of a (usual) groupoid  with natural transformations form a strict 2-group or, equivalently, a crossed module. 
 The crossed module involves the product of bisections and the composition of automorphisms, together with the action of automorphisms on bisections by conjugation. Bisections are the 2-arrows from the identity morphisms to automorphisms, and the composition of bisections can be viewed as the horizontal composition of 2-arrows. In the present paper this construction is extended to the Ehresmann--Schauenburg bialgebroid of a Hopf--Galois extension by constructing a crossed module for the bisections and the automorphisms of the bialgebroid. 

The paper is organised as follow: In \S 2 we recall all the relevant concepts and notation on noncommutative principal bundles (Hopf--Galois extensions), gauge groups and bialgebroids that we need. In \S 3, we first have Ehresmann--Schauenburg bialgebroids and the group of their bisections, then we show that when the base algebra belongs to the centre of the total space algebra, the gauge group of a noncommutative principal bundle  is group isomorphic to the group of bisections of the corresponding Ehresmann--Schauenburg bialgebroid. In \S 4 we consider Galois objects, which can be viewed as noncommutative principal bundles over a point. Several examples are studied here, such as Galois objects for a cocommutative Hopf algebra, in particular group Hopf algebras, regular Galois objects (Hopf algebras as self-Galois objects) and Galois objects of Taft algebras. In \S 5, we study the crossed module structure in terms of the bisection and the automorphism groups of an Ehresmann--Schauenburg bialgebroid.  When restricting to a Hopf algebra, we show that characters and automorphisms also form a crossed module structure; and this can generate the representation theory of 2-groups (or crossed modules) on Hopf algebras. We work out in details this construction for the Taft algebras.

%\noindent
%{\bf Acknowledgment:} 
%We are most grateful to Chiara Pagani for many useful discussions. GL was partially supported by INFN, Iniziativa Specifica GAST and INDAM-GNSAGA. 
%

\section{Algebraic preliminaries} 
We recall here known facts from algebras and coalgebras and corersponding modules and comodules.
We also recall the more general notions of rings and corings over an algebra as well as the associated notion of bialgebroid.
We move then to Hopf--Galois extensions, as noncommutative principal bundles and the definitions of gauge groups.

\subsection{Algebras, coalgebras and all that}\label{acat}
To be definite we work over the field $\IC$ of complex numbers but in the following this could be substituted by any field $k$. Algebras are assumed to be unital and associative with morphisms of algebras taken to be unital,
and co-algebras are assumed to be counital and coassociative with morphism of co-algebras taken to be co-unital.
For the coproduct of a coalgebra $\Delta : H \to H \ot H$ we use the Sweedler notation
$\Delta(h)=\one{h}\ot\two{h}$ (sum understood), and
its iterations: $\Delta^n=(\id \ot  \Delta_H) \circ\Delta_H^{n-1}: h \mapsto \one{h}\ot \two{h} \ot
\cdots \ot h_{\scriptscriptstyle{(n+1)\;}}$.
We denote by $*$ the convolution product in the dual vector space
$H':=\mathrm{Hom}(H,\IC)$, $(f * g) (h):=f(\one{h})g(\two{h})$.

Given an algebra $A$, a left $A$-module is a vector space $V$ carrying a left $A$-action, that is with a $\IC$-linear map
$\triangleright_V: A \ot  V \to V$ such that
\beq
(a b) \triangleright_V v = a \triangleright_V (b \triangleright_V v ) ~,\quad 1 \triangleright_V v = v \, .
\eeq
Dually, with a bialgebra $H$, a right $H$-comodule is a vector space $V$ carrying a
right $H$-coaction, that is with a $\IC$-linear map $\delta^V : V\to V\ot  H$
such that
\begin{align}
(\id\ot  \Delta)\circ \delta^V & = (\delta^V\ot  \id)\circ \delta^V ~,  \quad
(\id\ot  \epsilon) \circ \delta^V = \id ~. \label{eqn:Hcomodule}
\end{align}
In Sweedler notation, $v\mapsto \delta^V (v) = \zero{v}\ot  \one{v}$, and the right $H$-comodule properties
read,
\begin{align*}
\zero{v} \ot  \one{(\one{v})}\ot  \two{(\one{v})} &= \zero{(\zero{v})} \ot \one{(\zero{v})} \ot  \one{v} =: \zero{v} \ot\one{v} \ot \two{v} ~, \\
\zero{v} \,\epsilon (\one{v}) &= v ~,
\end{align*}
for all $v\in V$.
The $\IC$-vector space tensor product $V\ot  W$ of two $H$-comodules is a $H$-comodule with the right tensor product $H$-coaction
\begin{align}\label{deltaVW}
 \delta^{V\ot  W} :V\ot  W \longrightarrow  V\ot  W\ot  H~, \quad
 v\ot  w \longmapsto \zero{v}\ot  \zero{w} \ot
 \one{v}\one{w} ~.
\end{align}
a $H$-comodule map $\psi:V\to W$  between two $H$-comodules is a $\IC$-linear map
$\psi:V\to W$ which is $H$-equivariant (or $H$-colinear), that is, $\delta^W\circ \psi=(\psi\ot \id)\circ \delta^V$.

In particular, a right $H$-comodule algebra is an algebra  $A$  which is a right $H$-comodule such that the multiplication and unit of $A$ are morphisms of $H$-comodules.
This is equivalent to requiring the coaction $\delta^A: A\to A\ot  H$ to be
a morphism of unital algebras (where $A\ot  H$ has the usual tensor
product algebra structure).
Corresponding morphisms are $H$-comodule maps which are also algebra maps.

In the same way, a right $H$-comodule coalgebra is a coalgebra $C$ which is a right $H$-comodule and such that the
coproduct and the counit of $C$ are morphisms of $H$-comodules. Explicitly, this means that, for each $c \in C$,
\begin{align*}
\zero{(\one{c})} \ot \zero{(\two{c})} \ot \one{(\one{c})} \one{(\two{c})}
& =
\one{(\zero{c})} \ot \two{(\zero{c})} \ot \one{c} \, ,
\\
\epsilon(\zero{c}) \one{c} & =\epsilon(c) 1_H \, .
\end{align*}
Corresponding morphisms are $H$-comodule maps which are also coalgebra maps.
Clearly, there are right $A$-modules and left $H$-comodule versions of the above ones.

Next, let $H$ be a bialgebra and let $A$ be a right $H$-comodule algebra.
An $(A,H)$-relative Hopf module $V$ is a right $H$-comodule with a compatible left $A$-module structure.
That is  the left action $\triangleright_V: A\ot  V \to V$ is a morphism of $H$-comodules, 
$ \delta^V \circ \triangleright_V = (\triangleright_V \ot \id) \circ \delta^{A \ot V}$.
Explicitly, for all $a\in A$ and $v\in V$,
\beq\label{eqn:modHcov}
\zero{(a \triangleright_V v)} \ot \one{(a \triangleright_V v)} = \zero{a} \triangleright_V \zero{v} \ot \one{a}\one{v} ~.
\eeq
A morphism of
$(A,H)$-relative Hopf modules is a morphism of right $H$-comodules  which is also a morphism of left $A$-modules.
In a similar way one can consider the case for the algebra $A$ to be acting on the right, or  with a left and right 
$A$-actions.

We end this part on preliminaries by recalling the notions of bialgebroids (cf. \cite{Boehm}, \cite{BW}).

For an algebra $B$ a {\em $B$-ring} is a triple $(A,\mu,\eta)$. Here $A$ is a $B$-bimodule with $B$-bimodule maps
$\mu:A\ot_ {B} A \to A$ and $\eta:B\to A$, satisfying the following associativity
\begin{equation}\label{eq:cat.associ}
\mu\circ(\mu\ot _{B}  \id_A)=\mu\circ (\id_A \ot _{B} \mu)
\eeq
and unit conditions,
\beq
\mu\circ(\eta \ot _{B} \id_A)=\id_A=\mu\circ (\id_A\ot _{B} \eta).
\end{equation}
A morphism of $B$-rings $f:(A,\mu,\eta)\to (A',\mu',\eta')$ is an
$B$-bimodule map $f:A \to A'$, such that
$f\circ \mu=\mu'\circ(f\ot_{B} f)$ and $f\circ \eta=\eta'$.

From \cite[Lemma 2.2]{Boehm} there is a bijective correspondence between $B$-rings $(A,\mu,\eta)$ and algebra morphisms
$\eta : B \to A$. Starting with a $B$-ring $(A,\mu,\eta)$, one obtains a multiplication map $A \ot A \to A$ by composing the canonical surjection $A \ot A \to A\ot_B A$ with the map $\mu$. Conversely, starting with an algebra map $\eta : B \to A$, a $B$-bilinear associative multiplication $\mu:A\ot_ {B} A \to A$ is obtained from the universality of the coequaliser $A \ot A \to A\ot_B A$ which identifies an element $ a r \ot a'$ with $ a \ot r a'$.

Dually,  for an algebra $B$ a {\em $B$-coring} is a
triple $(C,\Delta,\epsilon)$. Here $C$ is a $B$-bimodule with $B$-bimodule maps
$\Delta:C\to C\ot_{B} C$ and $\epsilon: C\to B$, satisfying the following
coassociativity and counit conditions,
\begin{align}\label{eq:cat.coassoci}
(\Delta\ot _{B} \id_C)\circ \Delta = (\id_C \ot _{B} \Delta)\circ \Delta, \quad
(\epsilon \ot _{B} \id_C)\circ \Delta = \id_C =(\id_C \ot _{B} \epsilon)\circ \Delta.
\end{align}

A morphism of $B$-corings $f:(C,\Delta,\epsilon)\to
(C',\Delta',\epsilon')$ is a $B$-bimodule map $f:C \to C'$, such that
$\Delta'\circ f=(f\ot_{B} f)\circ \Delta$ and
$\epsilon' \circ f =\epsilon$.

Finally, let $B$ be an algebra.
A {\em left $B$-bialgebroid} $\cL$ consists of an $(B\ot B^{op})$-ring
together with a $B$-coring structures on the same vector space $\cL$ with mutual compatibility conditions.
From what said above, an $(B\ot B^{op})$-ring $\cL$ is the same as an algebra map $\eta : B \ot B^{op} \to \cL$.
Equivalently, one may consider the restrictions
$$
s := \eta ( \, \cdot \, \ot_B 1_B ) : B \to \cL \quad \mbox{and} \quad t := \eta ( 1_B \ot_B  \, \cdot \, ) : B^{op} \to \cL
$$
which are algebra maps with commuting ranges in $\cL$, called the \emph{source} and the \emph{target} map of the
$(B\ot B^{op})$-ring $\cL$. Thus a $(B\ot B^{op})$-ring is the same as a triple $(\cL,s,t)$ with $\cL$ an algebra and $s: B \to \cL$
and $t: B^{op} \to \cL$ both algebra maps with commuting range.

Thus, for a left $B$-bialgebroid $\cL$ the compatibility conditions are required to be
\begin{itemize}
\item[(i)] The bimodule structures in the $B$-coring $(\cL,\Delta,\epsilon)$ are
related to those of the $B\ot B^{op}$-ring $(\cL,s,t)$ via
\begin{equation}\label{eq:rbgd.bimod}
b\triangleright a \triangleleft b':= s(b) t(b')a \, \qquad \textrm{for} \,\, b, b'\in B, \, a\in \cL.
\end{equation}

\item[(ii)] Considering $\cL$ as a $B$-bimodule as in \eqref{eq:rbgd.bimod},
  the coproduct $\Delta$ corestricts to an algebra map from $\cL$ to
\begin{equation}\label{eq:Tak.prod}
\cL \times_{B} \cL := \left\{\ \sum\nolimits_j a_j\ot_{B} a'_j\ |\ \sum\nolimits_j a_jt(b) \ot_{B} a'_j =
\sum\nolimits_j a_j \ot_{B}  a'_j s(b), \,\, \mbox{for all}  \,\, b \in B\ \right\},
\end{equation}
where $\cL \times_{B} \cL$ is an algebra via component-wise multiplication.
\\
\item[(iii)] The counit $\epsilon : \cL \to B$ is a left character on the $B$-ring $(\cL,s,t)$, that is it satisfies the properties,
for $b\in B$ and $a,a' \in \cL$,
\begin{itemize}
\item[(1)] $\epsilon(1_{\cL})=1_{B}$,  \qquad (unitality)
\item[(2)] $\epsilon(s(b)a)=b\epsilon(a) $,  \qquad (left $B$-linearity)
\item[(3)] $\epsilon(as(\epsilon(a')))=\epsilon(aa')=\epsilon(at (\epsilon(a')))$,  \qquad (associativity) \, .
\end{itemize}
\end{itemize}

We finish this part with an additional concept that we shall use later on in Section \ref{cmhg}.
\begin{defi}\label{amoeba}
Let $(\cL, \Delta, \epsilon,s,t)$ be a left bialgebroid over the algebra $B$. 
An automorphism of the bialgebroid $\cL$ is a pair $(\Phi, \varphi)$
of algebra automorphisms, $\Phi: \cL \to \cL$, $\varphi : B \to B$ such that:

\begin{enumerate}[(i)]

\item $\Phi\circ s = s\circ \varphi $;
\item $\Phi\circ t = t\circ \varphi $;
\item $(\Phi\ot_{B} \Phi)\circ \Delta = \Delta \circ \Phi $;  
\item $\epsilon\circ \Phi= \varphi\circ \epsilon $.
 \end{enumerate}
\end{defi}
\noindent
In fact the map $\varphi$ is uniquely determined by $\Phi$ via $\varphi = \epsilon \circ \Phi \circ s$ 
and one can just say that $\Phi$ is a bialgebroid automorphism. Automorphisms of a bialgebroid $\cL$
form a group by composition that we simply denote $\Aut(\cL)$. 

\begin{rem}\label{newbimodule}
Here the pair of algebra maps $(\Phi, \varphi)$ can be viewed as a bialgebroid map (cf. \cite{kornel}, \S 4.1) between two copies of $\cL$ with different source and target map (and so $B$-bimodule structure). If $s, t$ are the source and target maps on $\cL$, one defines new source and target maps on $\cL$ 
by $s':=s\circ \varphi $ and $t':=t\circ \varphi $ with the new bimodule structure given by $b\triangleright_{\varphi}c\triangleleft_{\varphi}\tilde{b}:=s'(b)t'(\tilde{b})a$, for any $b, \tilde{b}\in B$ and $a\in \cL$ (see \eqref{eq:rbgd.bimod}).
Therefore we get a new left bialgebroid with product, unit, coproduct and counit not changed.

Clearly, from conditions (i) and (ii) $\Phi$ is a $B$-bimodule map: 
$\Phi(b \triangleright c\triangleleft \tilde{b})=b \triangleright_{\varphi}\Phi(c)\triangleleft_\varphi \tilde{b}$. The condition (iii) is well defined once the conditions (i) and (ii) are satisfied (the balanced tensor product in (iii) is induced by $s'$ and $t'$). Condition (iii) and (iv) imply $\Phi$ is a coring map, therefore $(\Phi, \varphi)$ is an isomorphism between the starting and the new bialgebroids.
\end{rem}

\subsection{Hopf--Galois extensions}\label{sec:hge}
In this section we will give a brief recall of Hopf--Galois extensions,
as noncommutative principal bundles. These are $H$-comodule algebras
$A$ with a canonically defined map $\chi: A\ot_B A\to A\ot H$
which is required to be invertible.

\begin{defi} \label{def:hg}
Let $H$ be a Hopf algebra and let $A$ be a $H$-comodule algebra with coaction $\delta^A$.
Consider the subalgebra $B:= A^{coH}=\big\{b\in A ~|~ \delta^A (b) = b \ot 1_H \big\} \subseteq A$ of coinvariant elements
and the corresponding balanced tensor product $A \ot_B A$.
The extension $B\subseteq A$ is called a $H$-\textup{Hopf--Galois extension} if the
\textit{canonical Galois map}
\begin{align}\label{canonical}  \chi := (m \ot \id) \circ (\id \ot _B \delta^A ) : A \ot _B A \longrightarrow A \ot H
~ ,
\quad a' \ot_B a  &\mapsto a' a_{\;(0)} \ot a_{\;(1)}
\end{align}
is bijective.
\end{defi}
\begin{rem}
In the following we shall  assume that for the Hopf Galois extension
$B\subseteq A$, the algebra $A$ is \emph{faithfully flat} as a left $B$-module.
This means, for instance, that taking the tensor product $\ot_B A$ with a sequence of right $B$-modules 
produces an exact sequence if and only if the original sequence is exact. 
\end{rem}

The canonical map $\chi$ is a morphism of relative Hopf modules for $A$-bimodules and right
$H$-comodules  (cf. \cite{ppca}).
Both $A \ot _B A$ and $A \ot H$ are $A$-bimodules.
The left $A$-module structures are the left multiplication on the first factors while the
right $A$-actions are:
$$
(a \ot_B a')a'':=  a \ot_B a'a'' \quad \mbox{and} \quad (a \ot h) a' := a \zero{a'} \ot h \one{a'}~.
$$
As for the $H$-comodule structure, the natural right tensor product $H$-coaction as in \eqref{deltaVW}:
\beq\label{AAcoact}
\delta^{A\ot  A}: A\ot  A\to A\ot  A\ot  H, \quad a\ot  a'  \mapsto  \zero{a}\ot  \zero{a'} \ot   \one{a}\one{a'} \,
\eeq
for all $a,a'\in A$, descends to the quotient $A\ot_B A$ because $B\subseteq A$ is the
subalgebra of $H$-coinvariants.
Similarly,  $A \ot H$ is endowed with the tensor product coaction, where   one regards the Hopf algebra $H$ as a right $H$-comodule
with the right adjoint $H$-coaction
\beq\label{adj}
\mathrm{Ad} :
h \longmapsto \two{h}\ot  S(\one{h})\,\three{h} ~.
\eeq
The right $H$-coaction on $A \ot H$
is then given, for all $ a\in A,~ h \in H$ by
\beq\label{AHcoact}
\delta^{A\ot  H}(a\ot  h) = \zero{a}\ot  \two{h} \ot  \one{a}\,S(\one{h})\, \three{h} \in A \ot H \ot H ~.
\eeq

Since the canonical Galois map $\chi$ is left $A$-linear, its inverse is
determined by the restriction $\tau:=\chi^{-1}_{|_{1_A \ot H}}$, named \textit{translation map},
\begin{eqnarray*}
\tau =\chi^{-1}_{|_{1_A \ot H}} :  H\to A\ot _B A ~ ,
\quad h \mapsto \tuno{h} \ot_B \tdue{h}~.
\end{eqnarray*}
The translation map enjoys a number of properties 
that we list here for later use. Firstly, it was shown in \cite[Prop. 3.6]{brz-tr} that,
$$
(\id \ot_B \delta^A)\circ \tau =(\tau \ot \id) \circ \Delta \, , \quad
 (  \tau  \ot  S) \circ \mbox{flip} \circ \Delta= (\id \ot \mbox{flip}) \circ (\delta^A \ot_B \id) \circ \tau  ~.
$$
On an element $h \in H$ these respectively read
\begin{align}
\tuno{h} \ot_B \zero{\tdue{h}} \ot \one{\tdue{h}} &= \tuno{\one{h}} \ot_B \tdue{\one{h}} \ot \two{h} \label{p4} ~, \\
\zero{\tuno{h}} \ot_B {\tdue{h}} \ot \one{\tuno{h}} &= \tuno{\two{h}}  \ot_B \tdue{\two{h}} \ot S(\one{h}) ~. \label{p1}
\end{align}
Furthermore, from \cite[Lemma 34.4]{BW}, for any $a\in A$ and $h, k\in H$, we have the following:
\begin{align}
\label{p7}
\tuno{h}\zero{\tdue{h}}\ot \one{\tdue{h}} &= 1_{A} \ot h ~, \\
\label{p5}
\tuno{h}\tdue{h} &= \epsilon(h)1_{A} ~, \\
\label{p2}
\tuno{(h k)}\ot_{B}\tdue{(h k)} &= \tuno{k}\tuno{h}\ot_{B}\tdue{h}\tdue{k} ~, \\
\label{p6}
\tuno{\one{h}}\ot_{B}\tdue{\one{h}}\tuno{\two{h}}\ot_{B}\tdue{\two{h}} & =\tuno{h}\ot_{B}1_{A}\ot_{B}\tdue{h} ~, \\
\label{p3}
\zero{a}\tuno{\one{a}}\ot_{B}\tdue{\one{a}} &= 1_{A} \ot_{B}a ~, 
\end{align}
for any $h, k\in H$ and $a\in A$.

Two $H$-Hopf--Galois extensions $A,A' $ of an 
algebra $B$ are isomorphic provided there exists an isomorphism of
$H$-comodule algebras $A \to A'$. This is the algebraic counterpart
for noncommutative principal bundles of the geometric
notion of isomorphism of principal $G$-bundles with fixed base
  space.
As in the geometric case this notion is relevant in the
classification of noncommutative principal bundles,
(cf. \cite{kassel-review}).

In the present paper we shall be mainly interested in \emph{Galois object}: given a Hopf algebra $H$, a $H$-Galois object 
is a 
$H$-Hopf--Galois extension of $\IC$. These could be thought of a noncommutative principal bundle over a point. It is well known
(cf. \cite{kassel-review}) that the set $\mbox{Gal}_H(\IC)$ of isomorphic classes of
$H$-Galois objects need not be trivial. This is in contrast to the fact that any (usual) fibre bundle over a point is trivial.

\subsection{The gauge groups}
In \cite{brz-tr} gauge transformations for a noncommutative principal bundles were defined to be
invertible and unital comodule maps, with no additional requirement. In particular they were not
asked to be algebra morphisms. A drawback of this approach is that the resulting gauge group might be
very big, even in the classical case; for example the gauge group of the a $G$-bundle over a point would be
much bigger than the structure group $G$. On the other hand, in \cite{pgc} gauge transformations were
required to be algebra homomorphisms. This implies in particular that they are invertible.

In the line of the paper \cite{pgc} we are lead to the following definition.

\begin{defi}\label{def:vgaugegroup}
Given a Hopf--Galois extension $B=A^{coH}\subseteq A$. Consider the collection
\begin{equation}\label{Acomm}
\Aut_H(A):=\{ F \in \Hom_{\A^H}(A, A) \, |  \,\, F{|_B} \in \Aut(B)\},
\end{equation}
of right $H$-comodule unital algebra morphisms of $A$ which restrict to automorphisms of $B$,
and the sub-collection
\begin{equation}
\label{Avcomm}
\Aut_{ver}(A) :=\{ F \in \Aut_H(A) \, |  \,\, F{|_B}=\id_{B}\}.
\end{equation}
of `vertical' ones, that is that in addition are left $B$-module morphisms.
\end{defi}
\noindent
Thus elements $F\in\Aut_H(A)$ preserve the (co)-action of the structure quantum group since they are such that
$\delta^A \circ F = (F \ot \id) \delta^A$ (or $\zero{F(a)} \ot \one{F(a)} = F(\zero{a}) \ot \one{a}$). 
And if in $\Aut_{ver}(A)$ they also preserve the base space algebra $B$.
These will be called the gauge group and the vertical gauge group respectively: in parallel with \cite[Prop. 3.6]{pgc}, $\Aut_H(A)$ and $\Aut_{ver}(A)$ are groups when $B$ is restricted to be in the centre of $A$
by the following proposition:

\begin{prop}\label{prop:gsogg} Let $B=A^{coH}\subseteq A$ be a $H$-Hopf--Galois extension with $B$ in the centre of $A$.
Then $\Aut_H(A)$ is a group with respect to the composition of maps
$$
F \cdot G:= G \circ F
$$
for all  $F , G \in\Aut_H(A)$.
For $F\in\Aut_H(A)$ its inverse $F^{-1} \in \Aut_H(A)$  is given by
\beq\label{inv-F2}
F^{-1}:=m\circ ((F{|_B})^{-1}\ot \id)\circ (m\ot \id )\circ (\id\ot F\ot_B\id)\circ (\id\ot \tau)\circ \delta^A\:
\eeq
where $\tau$ is the translation map, that is for all $a\in A$,
\beq\label{inv-F2b}
F^{-1}(a):= (F{|_B})^{-1} \big(\zero{a} F (\tuno{\one{a}})\big) \, \tdue{\one{a}} \, .
\eeq
In particular the vertical homomorphisms $\Aut_{ver}(A)$ form a subgroup of $\Aut_H(A)$.
\end{prop}
\begin{proof}
The group multiplication is clearly well defined with unit the identity map on $A$.
Next, we compute that somewhat `implicitly', $\zero{a}F(\tuno{\one{a}})\ot _{B}\tdue{\one{a}}\in B\ot _{B} A$. Indeed
\begin{align*}
(\delta^A \ot_{B} \id_{A})(\zero{a}F(\tuno{\one{a}})\ot _{B}\tdue{\one{a}})& = \zero{\zero{a}}\zero{F(\tuno{\one{a}})}\ot\one{\zero{a}}\one{F(\tuno{\one{a}})}\ot_{B}\tdue{\one{a}}\\
& = \zero{a}F(\zero{\tuno{\two{a}}})\ot \one{a}\one{\tuno{\two{a}}}\ot_{B}\tdue{\two{a}}\\
& = \zero{a}F(\tuno{\two{\two{a}}})\ot \one{a}S(\one{\two{a}})\ot_{B}\tdue{\two{\two{a}}}\\
& = \zero{a}F(\tuno{\one{a}})\ot 1_{H}\ot_{B}\tdue{\one{a}},
\end{align*}
where the 2nd step uses that $F$ is $H$-equivalent map, the 3rd step uses \eqref{p1}; since $\delta^A$ is right $B$--linear, everything is well defined.
Now, being $B$ the coinvariant subalgebra of $A$ for the coaction, we have the exact sequence,
$$
\begin{array}{*{9}c}
0 &\longrightarrow & B & \stackrel{i}{\longrightarrow}
& A & \stackrel{\delta^A - \id_{A}\ot \id_H}{\longrightarrow} & A\ot H & \longrightarrow & 0
\end{array}
$$
And, since $A$ is faithfully flat as left $B$-module, we also have exactness of the sequence,
$$
\setlength{\arraycolsep}{1pt}
  \begin{array}{*{9}c}
0 &\longrightarrow & B\ot_{B}A & \stackrel{i\, \ot_{B} \id_{A}}{\longrightarrow}
& A\ot_{B}A & \stackrel{(\delta^A - \id_{A} \ot \id_H)\ot_{B}  \id_{A}}{\longrightarrow} & A\ot H\ot_{B}A & \longrightarrow & 0 \, .
  \end{array}
$$
Thus, the equality $(\delta^A \ot_{B} \id_{A})(\zero{a}F(\tuno{\one{a}})\ot _{B}\tdue{\one{a}})
= \zero{a}F(\tuno{\one{a}})\ot 1_{H}\ot_{B}\tdue{\one{a}}$
shows that
$\zero{a}F(\tuno{\one{a}})\ot _{B}\tdue{\one{a}}\in B\ot _{B} A$ and
thus $F^{-1}$ in \eqref{inv-F2b} is well defined.
 Let us check that $F^{-1}$ is an algebra map:
\begin{align*}
F^{-1}(aa')& = (F{|_B})^{-1} \big( \zero{(aa')}F(\tuno{\one{(aa')}})\big) \, \tdue{\one{(aa')}}\\
& = (F{|_B})^{-1} \big(\zero{a}\zero{a'}F(\tuno{\one{a'}})F(\tuno{\one{a}})\big) \, \tdue{\one{a}}\tdue{\one{a}'}\\
& = F^{-1}(a)F^{-1}(a'),
\end{align*}
where the 2nd step uses \eqref{p2}, and the last step uses the fact that $\zero{a'}F(\tuno{\one{a'}})\ot_{B}\tdue{\one{a}'}\in B\ot_{B}A$ and $B$ belongs to the centre of $A$, thus $F^{-1}$ is an algebra map. Also for any $b\in B$,  $F^{-1}(b)=(F{|_B})^{-1}(b)$, so $F^{-1}{|_B}\in\Aut(B)$. Then, for any $a\in A$
\begin{align*}
F^{-1}(F(a))& = (F{|_B})^{-1} \big(\zero{F(a)}F(\tuno{\one{F(a)}})\big) \, \tdue{\one{F(a)}}\\
& = (F{|_B})^{-1} \big(F(\zero{a})F(\tuno{\one{a}})\big) \, \tdue{\one{a}}\\
& = (F{|_B})^{-1} \big(F(\zero{a}\tuno{\one{a}})\big) \, \tdue{\one{a}}\\
& = a,
\end{align*}
where the 2nd step uses the $H$-equivariance of $F$, and the last step uses \eqref{p3}. Finally,
\begin{align*}
F(F^{-1}(a))& = F \big( (F{|_B})^{-1}(\zero{a}F(\tuno{\one{a}}))\tdue{\one{a}} \big)\\
& = \zero{a}F(\tuno{\one{a}})F(\tdue{\one{a}})\\
& = a.
\end{align*}
Thus $F^{-1}$ is the inverse map of $F \in \Aut_H(A)$.
The map $F^{-1}$ is $H$-equivariant as well so it belongs to $\Aut_H(A)$. Indeed, for any $a\in A$ we have
\begin{align*}
\zero{a}\ot\one{a}=\zero{F(F^{-1}(a))}\ot  \one{F(F^{-1}(a))}=F(\zero{F^{-1}(a)})\ot \one{F^{-1}(a)},
\end{align*}
where the last step uses the $H$-equivariance of $F$. Applying $F^{-1}\ot \id_{H}$ on both sides of the last equation we get
\begin{align*}
F^{-1}(\zero{a})\ot\one{a}=\zero{F^{-1}(a)}\ot \one{F^{-1}(a)}
\end{align*}
so $F^{-1}$ is $H$-equivariant. We conclude that $\Aut_H(A )$ is a group.\\
As for the vertical automorphisms, clearly $\Aut_{ver}(A)$ is closed for map compositions and one sees that $F^{-1} \in \Aut_{ver}(A)$ when $F \in \Aut_{ver}(A)$. Thus $\Aut_{ver}(A)$ is also a group.
\end{proof}

\begin{rem}
A similar proposition was first given in \cite{pgc}, for a Hopf algebra $H$ which is a coquasitriangular  Hopf  algebra, and $A$ is a quasi-commutative $H$-comodule algebra. As a consequence, $B$ belongs to the centre of $A$. In the present paper, we only require $B$ to belongs to the centre of $A$ without assuming $H$ to be a coquasitriangular  Hopf  algebra.

For the sake of the present paper, where we are concerned mainly with Galois objects, and seek to study their  gauge groups with relations to bisections of suitable groupoids, there is no restriction in assuming that the base space algebra $B$ be in the centre.
 \end{rem}

\section{Ehresmann--Schauenburg bialgebroids}
To any Hopf--Galois extension $B=A^{co \, H}\subseteq A$ one associates a $B$-coring \cite[\S 34.13]{BW} and a bialgebroid \cite[\S 34.14]{BW}.
These can be viewed as a quantization of the gauge or Ehresmann groupoid that is associated to a principal fibre bundle (cf. \cite{KirillMackenzie}).

\subsection{Ehresmann corings} The coring can be given in two equivalent ways.
Let $B=A^{co \, H}\subseteq A$ be a Hopf--Galois extension with right coaction $\delta^{A} : A \to A \ot H$.
Recall the diagonal coaction \eqref{AAcoact}, given for all $a, a' \in A$ by
$$
\delta^{A\ot  A}: A\ot  A\to A\ot  A\ot  H, \quad a\ot  a'  \mapsto
\zero{a}\ot  \zero{a'} \ot   \one{a}\one{a'} \, ,
$$
with corresponding $B$-bimodule of coinvariant elements,
\beq
(A\ot A)^{coH} = \{a\ot  \tilde{a}\in A\ot  A \, ; \,\, \zero{a}\ot  \zero{\tilde{a}}\ot  \one{a}\one{\tilde{a}}=a\ot  \tilde{a}\ot  1_H \}. \label{ec2}
\eeq
\begin{lem}
Let $\tau$ be the translation map of the Hopf--Galois extension. Then the $B$-bimodule of coinvariant elements in \eqref{ec2} is the same as the $B$-bimodule,
\begin{equation}\label{ec1}
\mathcal{C} :=\{a\ot  \tilde{a}\in A\ot  A : \,\, \zero{a}\ot  \tau(\one{a})\tilde{a}=a\ot  \tilde{a}\ot _B 1_A\}.
\end{equation}
\end{lem}
\begin{proof}
Let $a\ot \tilde{a}\in (A\ot A)^{coH}$. By applying $(\id_{A}\ot \chi)$ on $\zero{a}\ot\tuno{\one{a}}\ot_{B}\tdue{\one{a}}\tilde{a}$, we get
\begin{align*}
    \zero{a}\ot \tuno{\one{a}}\zero{\tdue{\one{a}}}\zero{\tilde{a}}\ot \one{\tdue{\one{a}}}\one{\tilde{a}} & =\zero{a}\ot \zero{\tilde{a}}\ot \one{a}\one{\tilde{a}} \\ &=a\ot \tilde{a}\ot 1_H =a\ot \chi(\tilde{a}\ot_{B} 1_A) \\
    & = (\id_{A}\ot \chi) (a \ot \tilde{a} \ot_{B} 1_A),
\end{align*}
where the first step uses \eqref{p7}. This shows that $(A\ot A)^{coH}\subseteq \C$.\\
Conversely, let $a\ot \tilde{a}\in \C$. By applying $(\id_{A}\ot \chi^{-1})$ on $\zero{a}\ot\zero{\tilde{a}}\ot \one{a}\one{\tilde{a}}$ and using the fact that $\chi^{-1}$ is left $A$-linear and \eqref{p2}, we get
\begin{align*}
    \zero{a}\ot \zero{\tilde{a}} \tuno{\one{\tilde{a}}}\tuno{\one{a}}\ot_{B}\tdue{\one{a}}\tdue{\one{\tilde{a}}}&=\zero{a}\ot\tuno{\one{a}}\ot_{B}\tdue{\one{a}}\tilde{a}\\
    &=a\ot \tilde{a}\ot_{B} 1_A \\ &=(\id_{A}\ot \chi^{-1})(a\ot \tilde{a}\ot 1_H),
\end{align*}
where in the first step \eqref{p3} is used. This shows that $ \C \subseteq (A\ot A)^{coH}$. \end{proof}

We have then the following definition \cite[\S 34.13]{BW}.
\begin{defi}\label{def:ec}
Let $B=A^{co \, H}\subseteq A$ be a Hopf--Galois extension with translation map $\tau$.
If $A$ is faithfully flat as a left $B$-module, then the $B$-bimodule $\mathcal{C}$ in \eqref{ec1} is a $B$-coring with coring coproduct
\beq\label{copro}
\Delta(a\ot  \tilde{a}) = \zero{a}\ot  \tau(\one{a})\ot \tilde{a}
= \zero{a} \ot \tuno{\one{a}} \ot_B \tdue{\one{a}} \ot \tilde{a},
\eeq
and counit
\beq\label{counit}
\epsilon(a\ot  \tilde{a})=a\tilde{a}.
\eeq
\end{defi}
\noindent
By applying the map $m_A\ot \id_H$ to elements of \eqref{ec2}, it is clear that $a\tilde{a}\in B$.
The above $B$-coring is called the \textit{Ehresmann} or \textit{gauge coring}; we denote it $\C(A, H)$.

 Whenever the structure Hopf algebra $H$ has an invertible antipode, the Ehresmann coring can also be given as a cotensor product (see \cite{HM16}).
Indeed, let $H$ be a Hopf algebra with invertible antipode. And
let $B=A^{co \, H}\subseteq A$ be a $H$-Hopf--Galois extension, with right coaction $\delta^{A} : A \to A \ot H$,
$a \mapsto \delta^A (a) = \zero{a}\ot  \one{a}$. Via the inverse of $S$ one has also a left
$H$ coaction ${}^{A}\delta: A\to H\ot A$, ${}^{A}\delta(a):=S^{-1}(\one{a})\ot \zero{a}$. One also shows that
$B:= {}^{coH}A=\big\{b\in A ~|~ {}^{A}\delta (b) = 1_H \ot b \big\}$.
Using the left $B$-linearity of $\delta^A$ and the right $B$-linearity of ${}^{A}\delta$ one has a $B$-bimodule,
\begin{align}\label{ec4}
A \, \square \, {}^H\!\!A & = \ker (\delta^A \ot \id_A - \id_A \ot {}^{A}\delta )  \nn \\
& = \left\{a\ot  \tilde{a}\in A\ot  A : \,\, \zero{a}\ot\one{a}\ot \tilde{a}=a\ot S^{-1}(\one{\tilde{a}})\ot\zero{\tilde{a}} \right\}
\end{align}
\begin{lem}
The bimodule $A \, \square \, {}^H\!\!A$ is the same as the bimodules $\C$ and $(A\ot A)^{coH}$.
\end{lem}
\begin{proof}
Let $a\ot \tilde{a}\in\C$. Then, by applying $(\id_{A}\ot \id_H \ot m_{A})\circ (\id_{A}\ot {}^{A}\delta\ot \id_{A})$
on $\zero{a}\ot \tuno{\one{a}}\ot_{B}\tdue{\one{a}}\tilde{a} = a\ot \tilde{a}\ot_{B} 1_A$ we get, for the left hand side,
\begin{align*}
 \zero{a}\ot S^{-1}(\one{\tuno{\one{a}}})\ot \zero{\tuno{\one{a}}}\tdue{\one{a}}\tilde{a}&=\zero{a}\ot S^{-1}(S(\one{\one{a}}))\ot\tuno{\two{\one{a}}}\tdue{\two{\one{a}}}\tilde{a}\\
 &=\zero{a}\ot \one{a}\ot \tilde{a},
\end{align*}
using \eqref{p1} in the first step and \eqref{p5} in the second one. As for the right hand side, we get $a\ot S^{-1}(\one{\tilde{a}})\ot \zero{\tilde{a}}$. Thus $\zero{a}\ot \one{a}\ot \tilde{a} =a\ot S^{-1}(\one{\tilde{a}})\ot \zero{\tilde{a}}$, and
$a\ot \tilde{a}\in A \, \square \, {}^H\!\!A$.

Conversely, assume $a\ot \tilde{a}\in A \, \square \, {}^H\!\!A$.
By applying $(\id_{A}\ot \id_{A}\ot m_{A})\circ (\id_{A} \ot \tau \ot \id_{A})$ on
$\zero{a}\ot \one{a}\ot \tilde{a} =a\ot S^{-1}(\one{\tilde{a}})\ot \zero{\tilde{a}}$, we get
\begin{align}\label{equofec}
    \zero{a}\ot \tuno{\one{a}}\ot_{B} \tdue{\one{a}}\tilde{a}=a\ot \tuno{S^{-1}(\one{\tilde{a}})}\ot_{B}\tdue{S^{-1}(\one{\tilde{a}})}\zero{\tilde{a}}.
\end{align}
Now, using \eqref{p7} in the second step, we have
\begin{align*}
\chi(\tuno{S^{-1}(\one{\tilde{a}})}\ot_{B}\tdue{S^{-1}(\one{\tilde{a}})}\zero{\tilde{a}})&=\tuno{S^{-1}(\one{\tilde{a}})}\zero{\tdue{S^{-1}(\one{\tilde{a}})}}\zero{\zero{\tilde{a}}}\ot \one{\tdue{S^{-1}(\one{\tilde{a}})}}\one{\zero{\tilde{a}}}\\
&=\zero{\tilde{a}}\ot S^{-1}(\two{\tilde{a}})\one{\tilde{a}}=\tilde{a}\ot 1=\chi(\tilde{a}\ot_{B} 1_A),
\end{align*}
From this $\tuno{S^{-1}(\one{\tilde{a}})}\ot_{B}\tdue{S^{-1}(\one{\tilde{a}})}\zero{\tilde{a}}=\tilde{a}\ot_{B}1$
which, when substituting in the right hand side of \eqref{equofec} yields
$\zero{a}\ot \tuno{\one{a}}\ot_{B} \tdue{\one{a}}\tilde{a} = a \ot \tilde{a}\ot_{B} 1_A$. Thus $a\ot \tilde{a}\in \C$.
\end{proof}
\noindent
Finally the coproduct \eqref{copro} translates to the coproduct on $A \, \square \, {}^H\!\!A$ written as,
\beq\label{copro1}
\Delta(a\ot  \tilde{a})=a\ot  \tau(S^{-1}(\one{\tilde{a}}))\ot \zero{\tilde{a}},
\eeq

The Ehresmann coring of a Hopf--Galois extension is
in fact a bialgebroid, called the \textit{Ehresmann--Schauenburg bialgebroid} (cf. \cite[34.14]{BW}).
One see that $\C(A, H) = (A\ot A)^{coH}$ is a subalgebra of $A \ot A^{op}$; indeed, given
$a\ot  \tilde{a}, a' \ot  \tilde{a}' \in (A\ot A)^{coH}$, one computes
$\delta^{A\ot  A}(a a' \ot \tilde{a}' \tilde{a}) =
\zero{a} \zero{a'} \ot \zero{\tilde{a}'} \zero{\tilde{a}}  \ot \one{a}\one{a'} \one{\tilde{a}'} \one{\tilde{a}} =
\zero{a} a' \ot \tilde{a}' \zero{\tilde{a}}  \ot \one{a} \one{\tilde{a}} =
a a' \ot \tilde{a}' \tilde{a} \ot  1_H$.

\begin{defi}\label{def:reb}
Let $\C(A, H)$ be the coring of a Hopf--Galois extension
$B=A^{co \, H}\subseteq A$, with $A$ faithfully flat as a left $B$-module. Then $\C(A, H)$
is a (left) $B$-bialgebroid with product
\beq\label{pro}
(a\ot  \tilde{a})\bullet_{\C(A, H)}({a'}\ot  \tilde{a}')=a a'\ot \tilde{a}'\tilde{a} ,
\eeq
for all $a\ot  \tilde{a}, \, a' \ot  \tilde{a}' \in \C(A, H)$ (and unit $1\ot  1\in A\ot  A$). The target and the  source  maps are given by
\beq\label{sourcetarget}
t(b)=1\ot  b, \quad \textup{and} \quad s(b)=b\ot  1.
\eeq
\end{defi}
\noindent
We refer to \cite[34.14]{BW} for the checking that all defining properties are satisfied.

\subsection{The groups of bisections}
The bialgebroid of a Hopf--Galois extension can be viewed as a quantization (of the dualization)
of the classical gauge groupoid, recalled in Appendix \ref{gaugegroupoid}, of a (classical) principal bundle.
Dually to the notion of a bisection on the classical gauge groupoid there is the notion of a bisection on the
Ehresmann--Schauenburg bialgebroid. And in particular there are vertical bisections. These bisections correspond to automorphisms and vertical automorphisms (gauge transformations) respectively.

\begin{defi}\label{def:bisection}
Let $\C(A, H)$ be the left Ehresmann--Schauenburg bialgebroid associated to a Hopf--Galois extension
$B=A^{coH}\subseteq A$. A bisection of $\C(A, H)$ is a unital algebra map $\sigma: \C(A, H)\to B$, such that $\sigma\circ t=\id_{B}$ and
  $ \sigma\circ s\in\Aut(B)$.
\end{defi}

In general the collections of all bisections do not have additional structure. As a particular case that parallels
Proposition \ref{prop:gsogg} we have the following.
\begin{prop}\label{prop:gsob}
Consider the left Ehresmann--Schauenburg bialgebroid $\C(A, H)$ associated to a Hopf--Galois extension $B=A^{co
  H}\subseteq A$. If $B$ is contained in the centre of $A$, then the set of all bisections of $\C$ is a group,
  denoted $\B(\C(A, H))$, with product defined by
\begin{align}\label{mulbis1}
\sigma_{1}\ast \sigma_{2}(a\ot  \tilde{a}) & :=(\sigma_{2} \circ s) \big(\sigma_{1}(\zero{a}\ot  \tuno{\one{a}})\big)
\, \sigma_{2}(\tdue{\one{a}}\ot \tilde{a}), \nn \\
&\:=  \sigma_{2} \big(\sigma_{1}(\zero{a}\ot  \tuno{\one{a}}) \, \tdue{\one{a}}\ot \tilde{a} \big)\nn \\
&\: =  \sigma_{2} \big(\sigma_{1}(\one{(a\ot  \tilde{a})}) \, \two{(a\ot  \tilde{a})} \big)
\end{align}
for any bisections $\sigma_1$, $\sigma_2$ and any element $a\ot  \tilde{a}\in \C(A, H)$. The unit of this group is the counit of the bialgebroid. And for any bisection $\sigma$, its inverse is given by
\beq\label{invobis1}
\sigma^{-1}(a\ot \tilde{a})=(\sigma\circ s)^{-1} \big(a\sigma(\zero{\tilde{a}}\ot \tuno{\one{\tilde{a}}})\, \tdue{\one{\tilde{a}}}\big).
\eeq
Here $(\sigma\circ s)^{-1}$ is the inverse of  $\sigma\circ s\in\Aut(B)$.
\end{prop}
\begin{proof}
The second equality in \eqref{mulbis1} follows from the fact that bisections are taken to be algebra maps.
The expressions on the right hand side of \eqref{mulbis1} and \eqref{invobis1} are well defined.
For any bisection $\sigma$ and any $b\in B$, $a\in A$ the condition
$\sigma\circ t=\id_{B}$ yields:
\beq\label{test1}
\sigma(\zero{a}\ot  \tuno{\one{a}} \, b) \, \tdue{\one{a}} =
\sigma(\zero{a}\ot  \tuno{\one{a}}) \, b \,\tdue{\one{a}} \, .
\eeq
As for the multiplication in \eqref{mulbis1}: for bisections $\sigma_1$, $\sigma_2$ and any $b\in B$,
we have
\begin{align*}
\sigma_{1}\ast \sigma_{2}(s(b))=\sigma_{1}\ast \sigma_{2}(b\ot 1)=\sigma_2(s(\sigma_1(b\ot 1)))
=(\sigma_2\circ s)\circ(\sigma_1\circ s)(b).
 \end{align*}
Being both $\sigma_1\circ s$ and $\sigma_2\circ s$ automorphisms of $B$, we have $(\sigma_{1}\ast \sigma_{2})\circ s\in\Aut(B)$.
Similarly one shows that $\sigma_{1}\ast \sigma_{2}(t(b))=b$ for $b\in B$, that is $(\sigma_{1}\ast \sigma_{2})\circ t = \id_B$. Also,
the multiplication is associative: let $\sigma_1$, $\sigma_2$, $\sigma_3$ be bisections, and let
$a\ot \tilde{a}\in \C(A, H)$. From
\begin{align*}
((\Delta\ot_{B} \id_{\C(A, H)})\circ \Delta) (a\ot \tilde{a})=\zero{a}\ot  \tuno{\one{a}}\ot _{B}\tdue{\one{a}}\ot\tuno{\two{a}}\ot _{B}\tdue{\two{a}}\ot \tilde{a},
\end{align*}
we have (using in the second step that $\sigma_{3}\circ s$ is an algebra map):
\begin{align*}
((\sigma_1 & \ast \sigma_2) \ast \sigma_3)(a\ot \tilde{a}) \\
&= (\sigma_3 \circ s) \Big( (\sigma_2 \circ s) \big(\sigma_1(\zero{a}\ot  \tuno{\one{a}}) \big) \, \sigma_2(\tdue{\one{a}}\ot\tuno{\two{a}}) \Big)\, \sigma_3(\tdue{\two{a}}\ot \tilde{a})\\
&=(\sigma_3 \circ s) \Big( (\sigma_2 \circ s) \big(\sigma_1(\zero{a}\ot  \tuno{\one{a}}) \big)\Big) \,
(\sigma_3 \circ s) \big(\sigma_2(\tdue{\one{a}}\ot\tuno{\two{a}}) \big)\, \sigma_3(\tdue{\two{a}}\ot \tilde{a})\\
&=((\sigma_2\ast\sigma_3) \circ s) \big(\sigma_1(\zero{a}\ot  \tuno{\one{a}}) \big) \, (\sigma_2\ast \sigma_3 ) (\tdue{\one{a}}\ot \tilde{a})\\
&=(\sigma_1\ast (\sigma_2\ast \sigma_3))(a\ot \tilde{a}) \, .
\end{align*}
The assumption that $B$ belongs to the centre of $A$ implies that the product
$\sigma_{1}\ast \sigma_{2}$ is an algebra map. Indeed, for $a\ot\tilde{a}$ and $a'\ot\tilde{a}'\in \C$ one has,
\begin{align*}
&\sigma_{1}\ast\sigma_{2}(aa'\ot\tilde{a}'\tilde{a})\\
&=(\sigma_{2} \circ t)(\sigma_{1}(\zero{(aa)'}\ot  \tuno{\one{(aa')}})) \,
(\sigma_{2}(\tdue{\one{(aa')}}\ot \tilde{a}'\tilde{a}))\\
&=(\sigma_{2} \circ t)(\sigma_{1}(\zero{a}\zero{a'}\ot  \tuno{\one{a'}}\tuno{\one{a}})) \, (\sigma_{2}(\tdue{\one{a}}\tdue{\one{a'}}\ot \tilde{a}'\tilde{a}))\\
&=(\sigma_{2} \circ t)(\sigma_{1}(\zero{a}\ot  \tuno{\one{a}}))\, (\sigma_{2} \circ t)(\sigma_{1}(\zero{a'}\ot  \tuno{\one{a'}}))\, (\sigma_{2}(\tdue{\one{a'}}\ot \tilde{a}'))(\sigma_{2}(\tdue{\one{a}}\ot \tilde{a}))\\
&=(\sigma_{1}\ast \sigma_{2})(a\ot  \tilde{a}) \, (\sigma_{1}\ast \sigma_{2})(a'\ot  \tilde{a}')
 \end{align*}
The 2nd step uses \eqref{p2}, the 3rd step uses the fact that $\sigma_{1}$ and $\sigma_{2}$ are both algebra maps, the last step uses that $B$ belongs to the centre.

Thus $\sigma_{1}\ast \sigma_{2}$ is a well defined algebra map.
Next, we check $\epsilon$ is the unit of this multiplication. Firstly, since $B$ is taken to be contained in the centre of $A$
the counit $\epsilon$ is an algebra map. Indeed, for any $a\ot \tilde{a}\in \C(A, H)$,
\begin{align*}
\epsilon(aa'\ot\tilde{a}'\tilde{a})=aa'\tilde{a}'\tilde{a}=a'\tilde{a}'a\tilde{a}=\epsilon(a\ot\tilde{a})\epsilon(a'\ot\tilde{a}'),
\end{align*}
the 2nd step using that $B$ belongs to the centre.
Then,
\begin{align*}
\sigma\ast \epsilon \, (a\ot \tilde{a})&=(\epsilon \circ t) \big(\sigma(\zero{a}\ot  \tuno{\one{a}}) \big)\,
\epsilon(\tdue{\one{a}}\ot \tilde{a}) =(\epsilon \circ t) \big( \sigma(\zero{a}\ot  \tuno{\one{a}}) \big) \,
\tdue{\one{a}}\tilde{a}\\
&=(\epsilon \circ t) \big(\sigma(a\ot \tilde{a}) \big) \\ &=\sigma(a\ot \tilde{a}),
 \end{align*}
where the 3rd step uses the definition of $\C$. Similarly, for any $a\ot \tilde{a}\in \C(A, H)$:
\begin{align*}
\epsilon\ast \sigma(a\ot \tilde{a})&=(\sigma \circ t) \big(\epsilon(\zero{a}\ot  \tuno{\one{a}}) \big) \sigma(\tdue{\one{a}}\ot \tilde{a})\\
&=(\sigma \circ t) \big( \zero{a}\tuno{\one{a}} \big) \, \sigma(\tdue{\one{a}}\ot \tilde{a}) \\
& =\sigma(a\ot \tilde{a}),
 \end{align*}
and for the last equality we use $\zero{a}\tuno{\one{a}}\ot _{B} \tdue{\one{a}}=1\ot_{B} a$.
Thus $\epsilon$ is the unit.

Next, let us check that the inverse  of a bisection $\sigma$ as given in \eqref{invobis1}, is well defined.
The quantity $a\sigma(\zero{\tilde{a}}\ot \tuno{\one{\tilde{a}}})\tdue{\one{\tilde{a}}}$,
the argument of $(\sigma\circ s)^{-1}$ in \eqref{invobis1}, belongs to $B$. Indeed,
with $\delta^A$ the coaction as in \ref{def:hg}, one has
\begin{align*}
\delta^A(a\sigma(\zero{\tilde{a}}\ot \tuno{\one{\tilde{a}}})\tdue{\one{\tilde{a}}})&=\zero{a}\sigma(\zero{\tilde{a}}
\ot \tuno{\one{\tilde{a}}}) \zero{({\tdue{\one{\tilde{a}}}})} \ot \one{a}\one{({\tdue{\one{\tilde{a}}}})}\\
&=\zero{a}\sigma(\zero{\tilde{a}}\ot{\tuno{\one{\tilde{a}}}})\tdue{\one{\tilde{a}}}\ot\one{a}\two{\tilde{a}}\\
&=\zero{a}\sigma(\zero{\zero{\tilde{a}}}\ot{\tuno{\one{\zero{\tilde{a}}}}})\tdue{\one{\zero{\tilde{a}}}}\ot\one{a}\one{\tilde{a}}\\
&=a\sigma(\zero{\tilde{a}}\ot \tuno{\one{\tilde{a}}})\tdue{\one{\tilde{a}}}\ot 1_{H},
\end{align*}
where the 1st step uses that $\sigma$ is valued in $B$, the 2nd use \eqref{p4}, the last one uses \eqref{ec2}.

And for any $b\in B$, $\sigma^{-1}(s(b))=(\sigma\circ s)^{-1}(b)$, so $\sigma^{-1}\circ s = (\sigma\circ s)^{-1} \in \Aut(B)$;  also
$\sigma^{-1}(t(b))=(\sigma\circ s)^{-1}(\sigma(b\ot 1))=(\sigma\circ s)^{-1}((\sigma \circ s)(b)))=b$, so
$\sigma^{-1}\circ t=\id_{B}$. \\
Next, let us show $\sigma^{-1}$ is indeed the inverse of $\sigma$. For $a\ot\tilde{a}\in \C$, we have
\begin{align*}
(\sigma^{-1} & \ast\sigma)(a\ot \tilde{a}) \\
&=(\sigma \circ s)(\sigma^{-1}(\zero{a}\ot \tuno{\one{a}})) \, \sigma(\tdue{\one{a}}\ot\tilde{a})\\
&=(\sigma \circ s) \Big((\sigma\circ s)^{-1} \big(\zero{a}\sigma(\zero{\tuno{\one{a}}}\ot\tuno{\one{\tuno{\one{a}}}})\tdue{\one{\tuno{\one{a}}}} \big) \Big) \, \sigma(\tdue{\one{a}}\ot\tilde{a})\\
&=\zero{a}\sigma(\zero{\tuno{\one{a}}}\ot\tuno{\one{\tuno{\one{a}}}})\tdue{\one{\tuno{\one{a}}}} \, \sigma(\tdue{\one{a}}\ot\tilde{a})\\
&=\zero{a}\sigma\big(\tuno{\two{a}}\ot\tuno{S(\one{a})})\tdue{S(\one{a})} \, \sigma(\tdue{\two{a}}\ot\tilde{a})\\
&=\zero{a}\sigma(\tuno{\two{a}}\tdue{\two{a}}\ot\tilde{a}\tuno{S(\one{a})}\big) \, \tdue{S(\one{a})}\\
&=\zero{a}\tilde{a} \, \tuno{S(\one{a})}\tdue{S(\one{a})}\\
&=a\tilde{a}\\
&=\epsilon(a\ot\tilde{a}),
\end{align*}
where the 4th step uses \eqref{p1}, the 5th step uses that $B$ belongs to the centre of $A$, the 6th and 7th steps use \eqref{p5}. On the other hand,
\begin{align*}
(\sigma\ast\sigma^{-1})(a\ot\tilde{a})
&=(\sigma^{-1} \circ s) \big(\sigma(\zero{a}\ot\tuno{\one{a}}) \big) \, \sigma^{-1}(\tdue{\one{a}}\ot\tilde{a})\\
&=(\sigma\circ s)^{-1} \big(\sigma(\zero{a}\ot\tuno{\one{a}})\big) \, (\sigma\circ s)^{-1}(\tdue{\one{a}}\sigma(\zero{\tilde{a}}\ot\tuno{\one{\tilde{a}}})\tdue{\one{\tilde{a}}})\\
&=(\sigma\circ s)^{-1} \big(\sigma(\zero{a}\zero{\tilde{a}}\ot\tuno{\one{\tilde{a}}}\tuno{\one{a}}\big) \, \tdue{\one{a}}\tdue{\one{\tilde{a}}})\\
&=(\sigma\circ s)^{-1} \big(\sigma(\zero{(a\tilde{a})}\ot\tuno{\one{(a\tilde{a})}}\big) \, \tdue{\one{(a\tilde{a})}})\\
&=(\sigma\circ s)^{-1}(\sigma(a\tilde{a}\ot 1))\\
&=(\sigma\circ s)^{-1} \big( (\sigma\circ s)(a \tilde{a}) \big)\\
&=a\tilde{a}\\
&=\epsilon(a\ot\tilde{a}),
\end{align*}
where the second step uses $\sigma^{-1}(s(b))=(\sigma\circ s)^{-1}(b)$, the 3rd step uses that $B$ belongs to the centre of $A$, the 4th step uses \eqref{p2}, and the 5th step uses that $a\tilde{a}\in B$.

Finally, the map $\sigma^{-1}$ is an algebra map:
\begin{align*}
\sigma^{-1}(aa'\ot\tilde{a}'\tilde{a})&=(\sigma\circ s)^{-1}\Big(aa'\sigma\big(\zero{(\tilde{a}'\tilde{a})\big)}\ot\tuno{(\one{\tilde{a}'\tilde{a})}}\Big)\tdue{(\one{\tilde{a}'\tilde{a})}}\\
&=(\sigma\circ s)^{-1}\big(aa'\sigma(\zero{\tilde{a}'}\zero{\tilde{a}}\ot \tuno{\one{\tilde{a}}}\tuno{\one{\tilde{a}}'})\tdue{\one{\tilde{a}}'}\tdue{\one{\tilde{a}}} \big)\\
&=(\sigma\circ s)^{-1}\big(aa'\sigma(\zero{\tilde{a}}\ot\tuno{\one{\tilde{a}}})\sigma(\zero{\tilde{a}}'\ot\tuno{\one{\tilde{a}}'})\tdue{\one{\tilde{a}}'}\tdue{\one{\tilde{a}}}\big)\\
&=(\sigma\circ s)^{-1}\big(aa'\sigma(\zero{\tilde{a}'}\ot\tuno{\one{\tilde{a}}'})\tdue{\one{\tilde{a}}'}\sigma(\zero{\tilde{a}}\ot\tuno{\one{\tilde{a}}})\tdue{\one{\tilde{a}}}\big)\\
&=(\sigma\circ s)^{-1}\big(a'\sigma(\zero{\tilde{a}'}\ot\tuno{\one{\tilde{a}}'})\tdue{\one{\tilde{a}}'}a\sigma(\zero{\tilde{a}}\ot\tuno{\one{\tilde{a}}})\tdue{\one{\tilde{a}}}\big)\\
&=(\sigma\circ s)^{-1}\big(a'\sigma(\zero{\tilde{a}'}\ot\tuno{\one{\tilde{a}}'})\tdue{\one{\tilde{a}}'}\big)
(\sigma\circ s)^{-1}\big(a\sigma(\zero{\tilde{a}}\ot\tuno{\one{\tilde{a}}})\tdue{\one{\tilde{a}}}\big)\\
&=\sigma^{-1}(a\ot \tilde{a}) \, \sigma^{-1}(a'\ot \tilde{a}');
\end{align*}
the second step uses \eqref{p2}, the 3rd step uses $\sigma$ is an algebra map, the 5th one uses that the image of $\sigma$ and $a'\sigma(\zero{\tilde{a}'}\ot \tuno{\one{\tilde{a}'}})\tdue{\one{\tilde{a}'}}$ are in $B$, which is in the centre of $A$.
\end{proof}

\begin{rem}\label{remlin}
Having asked that bisections are algebra maps, they are $B$-linear in the sense of the coring bimodule structure in \eqref{eq:rbgd.bimod}. That is, for any bisection $\sigma$ and $b\in B$,
$$
\sigma\big((a\ot\tilde{a}) \triangleleft b\big) = \sigma\big(t(b) \bullet_{\C}  a\ot\tilde{a}   \big) = \sigma(a\ot\tilde{a}) \, \sigma(t(b))
= \sigma(a\ot\tilde{a}) \, b \,
$$
and
$$
\sigma\big( b \triangleright (a\ot\tilde{a})\big) = \sigma\big(s(b) \bullet_{\C} a\ot\tilde{a}   \big)
= \sigma(a\ot\tilde{a}) \, \sigma(s(b)) = \sigma(a\ot\tilde{a}) \, (\sigma \circ s)(b) \, .
$$
\end{rem}

Among all bisections an important role is played by the vertical ones.
\begin{defi}\label{def:vbisection}
Let $\C(A, H)$ be the left Ehresmann--Schauenburg bialgebroid associated to a Hopf--Galois extension $B=A^{co \, H}\subseteq A$. A \textit{vertical bisection} is a bisection of $\C$ which is also a left inverse for the target map $s$, that is
$ \sigma\circ s = \id_{B}$.
\end{defi}
Then the following statement is immediate.
\begin{cor}\label{prop:gsovb}
Consider the left Ehresmann--Schauenburg bialgebroid $\C(A, H)$ associated to a Hopf--Galois extension
$B=A^{co \, H}\subseteq A$. If $B$ is contained in the centre of $A$, then the set $\B_{ver}(\C(A, H))$ of all vertical
bisections of $\C$ is a group,
a subgroup of the group of all bisections $\B(\C(A, H))$, with the restricted product given by
\beq\label{mulbis}
\sigma_{1}\ast \sigma_{2}(a\ot  \tilde{a}) : = \sigma_{1}(\zero{a}\ot  \tuno{\one{a}})
\, \sigma_{2}(\tdue{\one{a}}\ot \tilde{a})
\eeq
for any vertical bisections $\sigma_1$, $\sigma_2$. Moreover, the inverse of a vertical bisection is given by
\beq\label{invobis}
\sigma^{-1}(a\ot \tilde{a}) = a\sigma(\zero{\tilde{a}}\ot \tuno{\one{\tilde{a}}})\, \tdue{\one{\tilde{a}}} =
\sigma(\zero{\tilde{a}}\ot \tuno{\one{\tilde{a}}})\, a\, \tdue{\one{\tilde{a}}} \, .
\eeq
\end{cor}
\begin{proof}
The right hand side of both \eqref{mulbis} and  \eqref{invobis} is seen to be a vertical bisection.
\end{proof}
\begin{rem}
We notice that the product \eqref{mulbis} on vertical bisections is just the convolution product due to the second expression for the coproduct in \eqref{copro},
\begin{align}\label{mulver}
\sigma_{1}\ast \sigma_{2}(a\ot  \tilde{a}) & = (\sigma_{1} \ot_B\sigma_{2}) \circ \Delta(a\ot  \tilde{a}) \nn \\
& = \sigma_{1}(\zero{a}\ot  \tuno{\one{a}}) \, \sigma_{2}(\tdue{\one{a}}\ot \tilde{a}).
\end{align}
We show directly that the product in \eqref{mulver} is well defined, Indeed, for $b\in B$ we have
\begin{align*}
\sigma_{1}(\zero{a}\ot  \tuno{\one{a}} \, b)  & \, \sigma_{2}(\tdue{\one{a}}\ot \tilde{a})
=\sigma_{1}\big( (\zero{a}\ot  \tuno{\one{a}}) \bullet_{\C} (1 \ot b) \big) \, \sigma_{2}(\tdue{\one{a}}\ot \tilde{a}) \\
&=\sigma_{1}(\zero{a}\ot  \tuno{\one{a}}) \, \sigma_{1}(1 \ot b) \, \sigma_{2}(\tdue{\one{a}}\ot \tilde{a}) \\
&=\sigma_{1}(\zero{a}\ot  \tuno{\one{a}}) \, b \, \sigma_{2}(\tdue{\one{a}}\ot \tilde{a}) \\
&=\sigma_{1}(\zero{a}\ot  \tuno{\one{a}}) \, \sigma_{2}(b \ot 1) \, \sigma_{2}(\tdue{\one{a}}\ot \tilde{a}) \\
&=\sigma_{1}(\zero{a}\ot  \tuno{\one{a}}) \, \sigma_{2}(b\, \tdue{\one{a}}\ot \tilde{a}) \,
\end{align*}
with the 4th step coming from $\sigma_2$ being vertical.
\end{rem}
\subsection{Bisections and gauge groups}
Recall the Definition~\ref{def:vgaugegroup} and the Proposition~\ref{prop:gsogg} concerning the gauge 
group of a Hopf--Galois extension. We have the following results.
\begin{prop}\label{atob}
Let $B=A^{coH}\subseteq A$ be a Hopf--Galois extension,
and let $\C(A, H)$ be the corresponding left Ehresmann--Schauenburg bialgebroid. If $B$ is in the centre of $A$, then there is a  group isomorphism $\alpha:\Aut_H(A)\to \B(\C(A, H))$.
The isomorphism $\alpha$ restricts to an isomorphism between vertical subgroups
$\alpha:\Aut_{ver}(A)\to \B_{ver}(\C(A, H))$.
\end{prop}
\begin{proof}
Let $F\in\Aut_H(A)$ and define $\sigma_{F}\in \B(\C(A, H))$ by
\beq\label{gtob}
\sigma_{F}(a\ot \tilde{a}):=F(a)\tilde{a},
\eeq
for any $a\ot\tilde{a}\in \C(A, H)$. This is well defined since
\begin{align*}
\delta^A(F(a)\tilde{a}) &= \zero{(F(a)\tilde{a})}\ot \one{(F(a)\tilde{a})}
\\
& =\zero{F(a)}\zero{\tilde{a}}\ot \one{F(a)}\one{\tilde{a}}=F(\zero{a})\zero{\tilde{a}}\ot\one{a}\one{\tilde{a}}\\
&=F(a)\tilde{a}\ot 1_{H},
\end{align*}
where the last equality use  \eqref{ec2}, thus $F(a)\tilde{a}\in B$. And $\sigma_{F}$ is an algebra map, since
\begin{align*}
\sigma_{F}((a'\ot  \tilde{a}')\bullet_{\C(A, H)}({a}\ot  \tilde{a}))&=\sigma_{F}(a' a\ot \tilde{a}\tilde{a}')
= F(a' a)\tilde{a}\tilde{a}'\\
&=F(a')F(a)\tilde{a}\tilde{a}'=F(a')(F(a)\tilde{a})\tilde{a}' =F(a')\tilde{a}'\sigma_{F}({a}\ot  \tilde{a})
\\
& =\sigma_{F}(a'\ot  \tilde{a}')\sigma_{F}({a}\ot  \tilde{a}),
\end{align*}
where the 5th equality uses that $B$ is in the centre of $A$. It is clear that $\sigma_{F}\circ t=\id_{B}$ and
$\sigma_F\circ s=F|_{B}\in\Aut(B)$. Thus $\sigma_{F}$ is a well defined bisection. By the definition \eqref{gtob},
\begin{align*}
\sigma_{\id_{A}}(a\ot\tilde{a})=a\tilde{a}=\epsilon(a\ot\tilde{a})
\end{align*}
and for any $a\ot \tilde{a}\in \C(A, H)$ we have
\begin{align*}
\sigma_{G}\ast\sigma_{F}(a\ot \tilde{a})&= (\sigma_{F} \circ s) \big(\sigma_{G}(\zero{a}\ot \tuno{\one{a}}) \big) \,
\sigma_{F}(\tdue{\one{a}}\ot \tilde{a})\\
&= \sigma_{F} \big(G(\zero{a})\tuno{\one{a}}\ot 1)\big) \, \sigma_{F}(\tdue{\one{a}}\ot \tilde{a})\\
&=F(G(\zero{a})\tuno{\one{a}})F(\tdue{\one{a}})\tilde{a}\\
&=F\big(G(\zero{a})\tuno{\one{a}}\tdue{\one{a}} \big) \, \tilde{a}\\
&=F(G(a)) \, \tilde{a}\\
&=\sigma_{(G\cdot F)}(a\ot \tilde{a}),
\end{align*}
where the 5th step uses \eqref{p5}.

Conversely, given a bisection $\sigma$, one can define an algebra map  $F_{\sigma}: A\to A$ by
\beq\label{btog}
F_{\sigma}(a):=\sigma(\zero{a}\ot\tuno{\one{a}})\tdue{\one{a}} \, .
\eeq
We have already seen (cf. \eqref{test1}) that the right hand side of \eqref{btog} is well defined.
Clearly $F_{\sigma}(b)=(\sigma \circ s)(b)$ for any $b\in B$, so $F_{\sigma}|_{B}\in\Aut(B)$.
Moreover,
\begin{align*}
F_{\sigma}(aa')&=\sigma(\zero{(aa')}\ot \tuno{\one{(aa')}})\tdue{\one{(aa')}}=\sigma(\zero{a}\zero{a'}\ot \tuno{(\one{a}\one{a'})})\tdue{(\one{a}\one{a'})}\\
&=\sigma(\zero{a}\zero{a'}\ot \tuno{\one{a'}}\tuno{\one{a}})\tdue{\one{a}}\tdue{\one{a'}}\\
&=\sigma(\zero{a'}\ot\tuno{\one{a'}})\sigma(\zero{a}\ot\tuno{\one{a}})\tdue{\one{a}}\tdue{\one{a'}}\\
&=F_{\sigma}(a)F_{\sigma}(a'),
\end{align*}
where in the third step we use \eqref{p2}. Also, $F_{\sigma}$ is $H$-equivalent:
\begin{align*}
\zero{F_{\sigma}(a)}\ot\one{F_{\sigma}(a)}&=\sigma(\zero{a}\ot\tuno{\one{a}})\zero{\tdue{\one{a}}}\ot\one{\tdue{\one{a}}}\\
&=\sigma(\zero{a}\ot\tuno{\one{a}})\tdue{\one{a}}\ot\two{a}\\
&=F_{\sigma}(\zero{a})\ot\one{a},
\end{align*}
where the 2nd step uses \eqref{p4}, thus $F_{\sigma}\in\Aut_H(A)$.

The map $\alpha$ is an isomorphism. Indeed for any $a\ot \tilde{a}\in \C(A, H)$
and any $\sigma\in\B(\C(A, H))$:
\begin{align*}
\sigma_{F_{\sigma}}(a\ot \tilde{a})=F_{\sigma}(a)\tilde{a}=\sigma(\zero{a}\ot\tuno{\one{a}})\tdue{\one{a}}\tilde{a}=\sigma(a\ot \tilde{a}),
\end{align*}
where the last step uses \eqref{ec1}. On the other hand, for any $a\in A$ and any $F\in\Aut_H(A)$:
\begin{align*}
F_{\sigma_{F}}(a)=\sigma_{F}(\zero{a}\ot\tuno{\one{a}})\tdue{\one{a}}=F(\zero{a})\tuno{\one{a}}\tdue{\one{a}}=F(a).
\end{align*}
Finally, for a vertical automorphism $F \in \Aut_{ver}(A)$, it is clear
that the corresponding $\sigma_F\in\B_{ver}(\C(A, H))$,
and conversely  from $\sigma \in \B_{ver}(\C(A, H))$ we have $F_\sigma\in\Aut_{ver}(A)$.
\end{proof}

\subsection{Extended bisections and gauge groups}\label{extended}
We have already mentioned that gauge transformations for a noncommutative principal
bundles could be defined without asking them to be algebra homomorphisms \cite{brz-tr}.
Mainly for the sake of completeness we record here a version of them via bialgebroid and bisections.
To distinguish them from the analogous concepts introduced in the previous section, and for lack of a better name, we call the extended gauge transformation and extended bisections.

In the same vein of \cite{brz-tr} we have the following definition.
\begin{defi}\label{def:gegaugegroup}
Given a Hopf--Galois extension $B=A^{coH}\subseteq A$. Its \textit{extended gauge group} $\Aut^{ext}_H(A)$ consists of invertible $H$-comodule unital maps $F: A \to A$ such that their restrictions $F{|_B} \in \Aut(B)$ and such that
$F(ba)=F(b)F(a)$ for any $b\in B$ and $a\in A$.
The \textit{extended vertical gauge group} $\Aut^{ext}_{ver}(A)$ is made of elements $F \in \Aut^{ext}_H(A)$ whose restrictions $F{|_B} = \id_B$. The group structure is map composition.
\end{defi}

In parallel with this we have then the  following.
\begin{defi}\label{def:gebisection}
Let $\C(A, H)$ be the left Ehresmann--Schauenburg bialgebroid of the Hopf--Galois extension
$B=A^{co \, H}\subseteq A$. An \textit{extended bisection} is a unital algebra map $\sigma: \C(A, H)\to B$,
such that $\sigma\circ t=\id_{B}$ and $\sigma\circ s\in\Aut(B)$, which in addition is $B$-linear in the sense of the $B$-coring structure on $\C$
(cf. Remark \ref{remlin}).
That is, $\sigma\big((a\ot\tilde{a})\triangleleft b\big)=\sigma(a\ot\tilde{a})\, b$, and
$\sigma\big(b\triangleright(a\ot\tilde{a})\big)=\sigma(a\ot\tilde{a}) \, (\sigma \circ s)(b)$.
The set of all extended bisections which are invertible for the product \eqref{mulbis1}:
\beq\label{mulbis1ex}
(\sigma_{1}\ast \sigma_{2}) (a\ot  \tilde{a}) :=(\sigma_{2} \circ s) \big(\sigma_{1}(\zero{a}\ot  \tuno{\one{a}})\big)
\, \sigma_{2}(\tdue{\one{a}}\ot \tilde{a}).
\eeq
will be denote by $\B^{ext}(\C(A, H))$, while
$\B^{ext}_{ver}(\C(A, H))$ will denote those which are invertible and vertical, that is such that
$\sigma\circ s = \id_B$ as well.
\end{defi}

\begin{lem}
The product \eqref{mulbis1ex} is well defined.
\end{lem}
\begin{proof}
We need to check that $\sigma_{1}\ast \sigma_{2}$ is $B$-linear in the sense of the definition.
Now, for any $a\ot a'\in \C$ and $b\in B$ we have
\begin{align*}
    (\sigma_{1}\ast\sigma_{2})((a\ot \tilde{a})\triangleleft b) &= (\sigma_{1}\ast\sigma_{2})((a\ot \tilde{a} b) \\
    &=(\sigma_{2} \circ s)\big(\sigma_{1}(\zero{a}\ot  \tuno{\one{a}})\big) \, (\sigma_{2}(\tdue{\one{a}}\ot \tilde{a}b))\\
    &=(\sigma_{2} \circ s)\big(\sigma_{1}(\zero{a}\ot  \tuno{\one{a}})\big) \, \sigma_{2} \big( (\tdue{\one{a}} \ot \tilde{a}) \triangleleft b) \big) \\
   &= (\sigma_{1}\ast\sigma_{2})(a\ot a') \, b.
\end{align*}
Similarly,
\begin{align*}
    (\sigma_{1}\ast\sigma_{2})(b\triangleright(a\ot \tilde{a}))&=(\sigma_{1}\ast\sigma_{2})( b a \ot \tilde{a} ) \\
    &=
    (\sigma_{2} \circ s) \big(\sigma_{1}( b \zero{a} \ot  \tuno{\one{a}}) \big) \, (\sigma_{2}(\tdue{\one{a}}\ot \tilde{a})) \\
    &=(\sigma_{2} \circ s) \big( \sigma_{1}(b\triangleright(\zero{a}\ot
    \tuno{\one{a}})) \big) \, (\sigma_{2}(\tdue{\one{a}}\ot \tilde{a}))\\
    &=(\sigma_{2} \circ s) \big( \sigma_{1}(\zero{a}\ot
    \tuno{\one{a}}) \, (\sigma_{1} \circ s) (b) \big)\, (\sigma_{2}(\tdue{\one{a}}\ot \tilde{a}))\\
    &=(\sigma_{2} \circ s) (\sigma_{1} \circ s) (b) \, (\sigma_{2} \circ s) \big(\sigma_{1}(\zero{a}\ot \tuno{\one{a}}) \big) \,
    \sigma_{2}(\tdue{\one{a}}\ot \tilde{a}) \\
    &= (\sigma_{1}\ast\sigma_{2})(a\ot \tilde{a}) \, \big( (\sigma_{1}\ast\sigma_{2}) \circ s\big) (b).
\end{align*}
Were the last step uses the identity $\big( (\sigma_{1}\ast\sigma_{2}) \circ s\big) (b) =
(\sigma_{2} \circ s) (\sigma_{1} \circ s) (b)$, and the last but one one the fact that  $\sigma \circ s \in \Aut(B)$ and that $B$ is in the centre.
\end{proof}

\begin{rem}
We remark that \eqref{invobis1} is now not the inverse for the product in \eqref{mulbis1ex} since,
in contrast to Proposition \ref{prop:gsob} we are not asking the bisections be algebra maps.
\end{rem}

Finally, in analogy with Proposition \ref{atob} we have the following.
\begin{prop}\label{lem:gatb}
Let $B=A^{coH}\subseteq A$ be a Hopf--Galois extension, and let $\C(A, H)$ be the corresponding left Ehresmann--Schauenburg bialgebroid. If $B$ belongs to the centre of $A$, there is a group isomomorphism
$\widehat{\alpha}: \Aut^{ext}_H(A)\to \B^{ext}(\C(A, H))$.
The isomorphism restricts to an isomorphism $\widehat{\alpha}_{v}: \Aut^{ext}_{ver}(A)\to
\B^{ext}_{ver}(\C(A, H))$ between vertical subgroups.
\end{prop}
\begin{proof}
This uses the same methods as Proposition \ref{atob}.
Given $F\in\Aut^{ext}_H(A)$, define its image as in \eqref{gtob}: $\sigma_{F}(a\ot \tilde{a})=F(a)\tilde{a}$.
Being $B$ in the centre, for all $b\in B$ we have,
\begin{align*}
\sigma_{F}( a\ot \tilde{a} b)& = F(a) \tilde{a} \, b = \sigma_{F}(a\ot \tilde{a}) \, b \, , \\
    \sigma_{F}(b a \ot \tilde{a})&=F(b a)\tilde{a} = F(b) F(a)\tilde{a} = \sigma_{F}(a\ot \tilde{a}) \, (\sigma_{F}\circ s)(b),
\end{align*}
that is, $\sigma_F\big((a\ot\tilde{a})\triangleleft b\big)=\sigma_F(a\ot\tilde{a})\, b$, and
$\sigma_F\big(b\triangleright(a\ot\tilde{a})\big)=\sigma_F(a\ot\tilde{a}) \, (\sigma_F \circ s)(b)$.
Conversely, for $\sigma\in \B^{ext}(\C(A, H))$, define its image as in \eqref{btog}:
$F_{\sigma}(a)=\sigma(\zero{a}\ot\tuno{\one{a}})\tdue{\one{a}}$.
Then $F_{\sigma}(ba)=\sigma(b\zero{a}\ot\tuno{\one{a}})\tdue{\one{a}}=(\sigma\circ s)(b)F_{\sigma}(a) =
F_{\sigma}(b) F_{\sigma}(a)$, due to $B$ in the centre of $A$. The rest of the proof goes as that of Proposition \ref{atob}.
(minus the algebra map parts).\end{proof}

\section{Bisections and gauge groups of Galois objects}

From now on we shall concentrate on \emph{Galois objects} of a Hopf algebra $H$.
These are noncommutative principal bundles over a point.
In contrast to the classical result that any fibre bundle over a point is trivial,
the set $\mbox{Gal}_H(\IC)$ of isomorphic classes of $H$-Galois objects need
not be trivial (cf. \cite{Bich}, \cite{kassel-review}). We shall illustrate later on this non-triviality with examples coming from group algebras and Taft algebras.

\subsection{Galois objects}
\begin{defi}\label{Galoisobject}
Let $H$ be a Hopf algebra, a $H$-\textit{Galois object} of $H$ is an
$H$-Hopf--Galois extension $A$ of the ground field $\IC$.
\end{defi}
Thus for a Galois object the coinvariant subalgebra is the ground field $\IC$.
Recall from Section \ref{acat} that
an $(A,H)$-relative Hopf module $M$ is a right $H$-comodule with a compatible right $A$-module structure.
That is the action is a morphism of $H$-comodules such that
$ \delta^M( ma) = \zero{m} \zero{a} \ot \one{m} \one{a}$ for all $a \in A$, $m\in M$. We have the following \cite{schau}:
\begin{lem}\label{baiso}
Let $M$ be an $(A,H)$-relative Hopf module. If $A$ is faithfully flat over $\IC$, the multiplication induces an isomorphism
$$
M^{co \, H} \ot A \to M,
$$
whose inverse is $M \ni m \mapsto \zero{m} \tuno{\one{m}} \ot \tdue{\one{m}} \in M^{co \, H} \ot A$.
\end{lem}
With coaction $\delta^A : A \to A\ot H$, $\delta^A(a) = \zero{a}\ot\one{a}$,
and translation map $\tau : H \to A \ot A$, $\tau(h)=\tuno{h}\ot\tdue{h}$,
for the Ehresmann--Schauenburg bialgebroid of a Galois object, being $B=\IC$ one has
(see also \cite[Def. 3.1]{schau}):
\begin{align}
\C(A, H)
& = \{a\ot  \tilde{a}\in A\ot  A \, : \,\, \zero{a}\ot  \zero{\tilde{a}}\ot  \one{a}\one{\tilde{a}}=a\ot  \tilde{a}\ot  1_H \} \label{ec2c} \\
&=\{a\ot  \tilde{a}\in A\ot  A : \,\, \zero{a}\ot \tuno{\one{a}} \ot\tdue{\one{a}} \tilde{a} = a\ot \tilde{a} \ot 1_A\} \label{ec1c}  \, .
\end{align}
The coproduct \eqref{copro} and counit \eqref{counit} become $\Delta_\C(a\ot  \tilde{a}) = \zero{a} \ot \tuno{\one{a}} \ot \tdue{\one{a}} \ot \tilde{a}$,
and $\epsilon_\C(a\ot  \tilde{a})=a\tilde{a} \in \IC$ respectively, for any $a\ot\tilde{a}\in \C(A, H)$. But now there is also an antipode \cite[Thm. 3.5]{schau} given, for any $a\ot\tilde{a}\in \C(A, H)$, by
\beq\label{anti}
S_{\C}(a\ot \tilde{a}):=\zero{\tilde{a}}\ot \tuno{\one{\tilde{a}}}a\tdue{\one{\tilde{a}}} \, .
\eeq
 Thus the Ehresmann--Schauenburg bialgebroid of a Galois object is a Hopf algebra.

 Now, given that $\C(A, H)=(A\ot A)^{coH}$, Lemma \ref{baiso} yields an isomorphism
 \beq\label{baiso-b}
 A \ot A \simeq \C(A, H) \ot A \, , \qquad \widetilde{\chi}(a\ot  \tilde{a}) = \zero{a} \ot \tuno{\one{a}} \ot \tdue{\one{a}} \tilde{a} \, .
 \eeq
 We finally collect some results of \cite{schau} (cf. Lemma 3.2 and Lemma 3.3) in the following:
 \begin{lem}\label{universal}
 Let $H$ be a Hopf algebra, and $A$ a (faithfully flat) $H$-Galois object of $H$.
 There is a right $H$-equivariant algebra map $\delta^{\C} : A \to \C(A, H) \ot A$ given by
 $$
\delta^{\C}(a) = \zero{a} \ot \tuno{\one{a}} \ot \tdue{\one{a}}
$$
 which is universal in the following sense: Given an algebra $M$ and a $H$-equivariant algebra map
 $\phi : A \to M \ot A$, there is a unique algebra map $\Phi :  \C(A, H) \to M$ such that
 $\phi = (\Phi \ot \id_A) \circ \delta^{\C}$. Explicitly, $\Phi(a \ot  \tilde{a}) \ot 1_A = \phi(a) \tilde{a}$.
\end{lem}

The ground field $\IC$ being undoubtedly in the centre,
for the bisections of the Ehresmann--Schauenburg bialgebroid $\C(A, H)$ of a Galois object $A$, we can use all results of previous sections. Clearly, any bisection of $\C(A, H)$ is now vertical as it is vertical any automorphism of the principal bundle  $A$. In fact bisections, being algebra maps, are just characters of the Hopf algebra
$\C(A, H)$ with convolution product in \eqref{mulbis} and inverse in \eqref{invobis} that, with the antipode in \eqref{anti}
can be written as $\sigma^{-1} = \sigma \circ S_{\C}$, as is the case for characters. From Proposition \ref{atob} we have then the isomorphism
\beq\label{isogen}
\Aut_{H}(A) \simeq \B(\C(A, H)) = \mbox{Char}(\C(A, H)).
\eeq
As for extended bisections and automorphisms as in Section \ref{extended} we have analogously from Proposition \ref{lem:gatb} the isomorphism,
\beq\label{isogenext}
\Aut^{ext}_H(A) \simeq \B^{ext}(\C(A, H)) = \mbox{Char}^{ext}(\C(A, H)) \, ,
\eeq
with $\mbox{Char}^{ext}(\C(A, H))$ the group of convolution invertible unital maps
$ \phi: \C(A, H) \to \IC$.

\subsection{Hopf algebras as Galois objects}
Any Hopf algebra $H$ is a $H$-Galois object with coaction given by its coproduct.
Then $H$ is isomorphic to the corresponding left bialgebroid $\C(H, H)$.

Let $H$ be a Hopf algebra with coproduct $\Delta(h)=\one{h}\ot\two{h}$. For the corresponding coinvariants: 
$\one{h}\ot\two{h}=h\ot 1$, we have $\epsilon(\one{h})\ot\two{h}=\epsilon(h)\ot 1$, this imply $h=\epsilon(h)\in \IC$ and
$H^{co \, H}=\IC$. 
Moreover, the canonical Galois map $\chi :  g \ot h \mapsto  g \one{h} \ot \two{h}$ is bijective with inverse
given by  
$\chi^{-1}( g \ot h ):=  g \, S(\one{h})\ot\two{h}$. Thus $H$ is a $H$-Galois object.

With $A=H$, the corresponding left bialgebroid becomes
\begin{align}
\C(H, H) & = \{ g \ot  h \in H \ot H : \,\, \one{g} \ot \one{h} \ot \two{g} \two{h} = g \ot h \ot  1_H \} \nn \\
& = \{g \ot h \in H\ot H : \,\, \one{g} \ot S(\two{g}) \ot \three{g} h = g \ot h \ot 1_A \} . \label{biHH} 
\end{align}
We have a linear map $\phi:\C(H, H) \to H$ given by $\phi(g\ot h):=g \, \epsilon(h)$. 
The map $\phi$ has inverse $\phi^{-1}: H\to\C(H, H)$, defined by $\phi^{-1}(h):=\one{h}\ot S(\two{h})$. 
This is well defined since 
\begin{align*}
\Delta^{H \ot H}(\one{h}\ot S(\two{h}))=\one{h}\ot S(\four{h})\ot \two{h}S(\three{h})=\one{h}\ot S(\two{h})\ot 1_{H}, 
\end{align*}
showing that $\one{h}\ot S(\two{h}) \in\C(H, H)$. 
 Moreover,  
 \begin{align*}
\phi(\phi^{-1}(h))=\phi(\one{h}\ot S(\two{h}))=h ,
\end{align*}
and
\begin{align*}
\phi^{-1}(\phi(g\ot h)) = \epsilon(h) \, \phi^{-1}(g) = \epsilon(h) \, \one{g}\ot S(\two{g}) = g \ot h .
\end{align*}
Here the last equality is obtained from the condition $\one{g}\ot S(\two{g})\ot \three{g}h=g\ot h\ot 1_{H}$ 
(for any $g\ot h\in\C(H, H)$, as in the second line of \eqref{biHH}) by 
applying $\id_{H} \ot \id_{H} \ot \epsilon $ on both sides and then multiplying the second and third factors:
$$
\one{g}\ot S(\two{g}) \, \epsilon(\three{g}) \epsilon(h) = g \ot h \, \epsilon(1_H) 
\quad \Longrightarrow \quad \epsilon(h) \, \one{g}\ot S(\two{g}) = g \ot h .
$$
The map $\phi$ is an algebra map: 
\begin{align*}
\phi((g \ot  h) \bullet_{\C} (g' \ot  h')) & =
\phi (g g' \ot  h' h) = gg'  \epsilon(h') \epsilon(h) \\ &= \phi(g \ot  h) \bullet_{\C} \phi(g' \ot  h').
\end{align*}  
It is also a coalgebra map: 
\begin{align*}
(\phi\ot\phi)(\Delta_{\C}(g\ot h))&=(\phi\ot\phi)(\one{g}\ot \tuno{\two{g}}\ot \tdue{\two{g}}\ot h)\\
&=(\phi\ot\phi)(\one{g}\ot S(\two{g})\ot \three{g}\ot h)\\
&=\one{g}\ot\two{g} \, \epsilon(h)\\
&=\Delta_{H}(\phi(g\ot h));
\end{align*}
\begin{align*}
\epsilon_{\C}(g\ot h) = gh = \epsilon_{H}(gh)=\epsilon_{H}(g)\epsilon_{H}(h)=\epsilon_{H}(\phi(g\ot h)) \, .
\end{align*}

\subsection{Cocommutative Hopf algebras}
We start with a class of examples coming from cocommutative Hopf algebras.
From \cite{schau} (Remark 3.8 and Theorem 3.5.) we have:
\begin{lem}\label{Go}
Let $H$ be a cocommutative Hopf algebra, and let $A$ be a $H$-Galois object.
Then the bialgebroid $\C(A, H)$ is isomorphic to $H$ as Hopf algebra.
\end{lem}
\begin{proof}
We give a sketch of the proof that uses Lemma \ref{universal}.
Start with the coaction $\delta^A : A \to A\ot H$, $\delta^A(a) = \zero{a}\ot\one{a}$, and translation map $\tau(h)=\tuno{h}\ot\tdue{h}$. 
Firstly, the image of $\tau$ is in $\C(A, H)$; indeed, for any $h\in H$, we get
\begin{align*}
\zero{\tuno{h}}\ot\zero{\tdue{h}}\ot\one{\tuno{h}}\one{\tdue{h}}&=\zero{\tuno{\one{h}}}\ot\tdue{\one{h}}\ot\one{\tuno{\one{h}}}\two{h}\\
&=\tuno{\two{\one{h}}}\ot \tdue{\two{\one{h}}}\ot S(\one{\one{h}})\two{h}\\
&=\tuno{h}\ot\tdue{h}\ot 1_{H},
\end{align*}
where the first step uses \eqref{p4}, and the second step uses \eqref{p1}. While $\tau$ is not an algebra map, being $H$ cocommutative, it is a coalgebra map.
Indeed, for any $h\in H$, 
\begin{align*}
\Delta_{\C}(\tau(h))&=\zero{\tuno{h}}\ot \tau(\one{\tuno{h}}) \ot\tdue{h}\\
&=\tuno{\two{h}} \ot \tau(S(\one{h})) \ot \tdue{\two{h}}\\
&=\tuno{\two{h}}\ot\tdue{\two{h}}\tuno{\three{h}} \, \tau(S(\one{h})) \, \tdue{\three{h}}\tuno{\four{h}}\ot\tdue{\four{h}}\\
&=\tuno{\three{h}}\ot\tdue{\three{h}}\tuno{\two{h}} \, \tau(S(\one{h})) \, \tdue{\two{h}}\tuno{\four{h}}\ot\tdue{\four{h}}\\
&=\tuno{\one{h}}\ot\tdue{\one{h}} \, \ot \tuno{\two{h}} \ot\tdue{\two{h}}\\
&=(\tau\ot\tau)(\Delta_{H}(h)) \, ,
\end{align*}
where the 2nd step uses \eqref{p1}: $\tuno{\two{h}} \ot \tdue{\two{h}} \ot S(\one{h}) = \zero{\tuno{h}} \ot {\tdue{h}}  \ot \one{\tuno{h}}$, 
the 3nd step uses twice \eqref{p6}: $\tuno{\one{h}}\ot\tdue{\one{h}}\tuno{\two{h}}\ot\tdue{\two{h}}=\tuno{h}\ot1_{A}\ot\tdue{h}$;
the 4rd step uses $H$ is cocommutative: we change the lower indices 2 and 3; 
and the 5th one uses: $\epsilon(h) \ot 1_A=\tau(S(\one{h})\two{h})=\tuno{\two{h}} \, \tau(S(\one{h}))\, \tdue{\two{h}}$.
Also,
\begin{align*}
\epsilon_{\C}(\tau(h))=\tuno{h}\tdue{h}=\epsilon(h)1_{A}. 
\end{align*} 
On the other hand, since $H$ is cocommutative, $A$ is also a left $H$-Galois object with coaction 
$\delta_L(a) = \one{a} \ot \zero{a}$ and bijective canonical map
$\chi_L(a\ot \tilde{a})=\one{a} \ot \zero{a} \tilde{a}$. 
The corresponding translation map is then $\tau_L = \tau \circ S$ where $S = S^{-1}$ (since $H$ is cocommutative) is the antipode of $H$. The map $\tau_L$ is a coalgebra map being the composition of two such maps (for $S$ this is the case again due to $H$ cocommutative). 

From the universality of Lemma \ref{universal},
there is a unique algebra map $\Phi :  \C(A, H) \to H$ such that $\delta_L = (\Phi \ot \id_A) \circ \delta^{\C}$; 
where $\delta^{\C} : H \to \C(A, H) \ot H$ as in the lemma. 
Explicitly,
$\Phi(a \ot  \tilde{a}) \ot 1_A = \delta_L(a) \tilde{a} = {\chi_L}_{|_\C} (a\ot \tilde{a})$ for $a\ot \tilde{a}\in \C(A, H)$. Indeed, with the isomorphism
$\widetilde{\chi}$ in \eqref{baiso-b}, the map $\Phi$ is such that $\chi_L = (\Phi \ot \id_A) \circ \widetilde{\chi}$, thus is an isomorphism since $\chi_L$ and $\widetilde{\chi}$ are such. The map $\Phi$ has inverse $\Phi^{-1} = \tau_L $ and thus is a coalgebra map. 
%On the other hand, since $H$ is cocommutative, $A$ is also a left $H$-Galois object with coaction $\delta_L(a) = \one{a} \ot \zero{a}$ and bijective canonical map
%$\chi_L(a\ot \tilde{a})=\one{a} \ot \zero{a} \tilde{a}$. From the universality of Lemma \ref{universal},
%there is a unique algebra map $\Phi :  \C(A, H) \to H$ such that $\delta_L = (\Phi \ot \id_A) \circ \delta^{\C}$; explicitly,
%$\Phi(a \ot  \tilde{a}) \ot 1_A = \delta_L(a) \tilde{a} = {\chi_L}_{|_\C} (a\ot \tilde{a})$ for $a\ot \tilde{a}\in \C(A, H)$. Indeed, with the isomorphism
%$\widetilde{\chi}$ in \eqref{baiso-b}, the map $\Phi$ is such that $\chi_L = (\Phi \ot \id_A) \circ \widetilde{\chi}$, thus is an isomorphism since $\chi_L$ and $\widetilde{\chi}$ are such. The map $\Phi$ respects co-structures and antipode.
%It inverse is found to be $\Phi^{-1} = \tau_L = \tau \circ S$ where $S = S^{-1}$ (since $H$ is cocommutative) is the antipode of $H$. The map $\tau_L$ is a coalgebra map being the composition of two such maps (for $S$ this is the case again due to $H$ cocommutative). 
\end{proof}

\noindent
Consequently, the isomorphisms \ref{isogen} and \ref{isogenext} for a cocommutative Hopf algebra $H$ are:
\beq\label{isogenH}
\Aut_{H}(A) \simeq \B(\C(A, H)) = \mbox{Char}(H)
\eeq
and
\beq\label{isogenextH}
\Aut^{ext}_H(A) \simeq \B^{ext}(\C(A, H)) = \mbox{Char}^{ext}(H) \, ,
\eeq
with $\mbox{Char}^{ext}(H)$ the group of convolution invertible unital maps $\phi: H \to \IC$ and $\mbox{Char}(H)$
the subgroup of those which are algebra maps (the characters of $H$).

\subsection{Group Hopf algebras} \label{gradedalg} 
Let $G$ be a  group, with neutral element $e$, and $H=\IC[G]$ be its group algebra. 
Its elements are finite sums $\sum \lambda_g \, g$ with $\lambda_g$ complex number. We assume that $\{g\, ,  g \in G\}$ is a vector space basis. The product in $\IC[G]$ follows from the group product in $G$, with algebra unit $1_{\IC[G]}=e$. The coproduct, counit, and antipode, making $\IC[G]$ a Hopf algebra are defined by $\Delta(g) = g \ot g$, $\epsilon(g)=1$, $S(g)=g^{-1}$. 

An algebra $A$ 
is $G$-graded, that is $A=\oplus_{g \in G} \, A_g$ and $A_g A_{h} \subseteq A_{gh}$ for all $g,h \in G$, 
if and only if $A$ is a right $\IC[G]$-comodule  algebra with coaction 
$\delta^A: A \to A \ot \IC[G]$, $a \mapsto \sum a_g \ot g $ for $a= \sum a_g$, $a_g \in A_g$.
Moreover,
the algebra $A$ is strongly $G$-graded, that is $A_g A_h =A_{gh}$,  
if and only if $A_e=A^{co\, \IC[G]} \subseteq A$ is Hopf--Galois (see e.g. \cite[Thm.8.1.7]{mont}).
Thus $\IC[G]$-Hopf--Galois extensions are the same as $G$-strongly graded algebras.  

In particular, if $A$ is a $\IC[G]$-Galois object, that is $A_e=\IC$, each component $A_g$ 
is one-dimensional. If we pick a non-zero element $u_g$ in each $A_g$, 
the multiplication of $A$ is determined by the products $u_g u_h$ for each pair $g, h$ of elements of G. 
We then have 
\beq\label{mult}
u_{g}u_{h}=\lambda(g, h) \, u_{gh}
\eeq 
for a non vanishing $\lambda(g, h)\in \IC$. We get then a map
$\lambda : G \times G \to \IC^{\times}$ which is in fact a two cocycle $\lambda \in H^2(G, \IC^{\times})$. 
Indeed, associativity of the product requires that $\lambda$ satisfies a 2-cocycle condition, that is for any $g, h\in G$, 
\beq\label{two-coc}
\lambda(g, h)\lambda(gh, k)=\lambda(h, k)\lambda(g, hk). 
\eeq
If we choose a different non-zero element $v_g \in A_g$, we shall have $v_g = \mu(g) u_g$, 
for some non-zero $\mu(g) \in \IC$. The multiplication \eqref{mult} will become $v_{g}v_{h}=\lambda'(g, h)v_{gh}$ with 
\beq
\lambda'(g, h) = \mu(g) \mu(h) (\mu(gh))^{-1} \lambda (g, h), 
\eeq
that is the two 2-cocycles $\lambda'$ and $\lambda$ are cohomologous. 
It is easy to check that for any map $\mu(g) : G \to \IC^{\times}$ the assignment $(g,h) \mapsto \mu(g) \mu(h) (\mu(gh))^{-1}$, is a coboundary, that is a `trivial' 2-cocycle which is cohomologous to $ \lambda (g, h) = 1$. 
Thus the multiplication in $A$ depends only on the second cohomology class of $\lambda \in H^2(G, \IC^{\times})$, the second cohomology group of $G$ with values in $\IC^{\times}$. We conclude that the equivalence classes of  $\IC[G]$-Galois objects are in bijective correspondence with the cohomology group $H^2(G, \IC^{\times})$.
\begin{exa}
From \cite[Ex. 7.13]{kassel-review} we have the following. 
For any cyclic group $G$ (infinite or not) one has $H^2(G, \IC^{\times})=0$. Thus any corresponding 
$\IC[G]$-Galois object is trivial.
On the other hand, for the free abelian group of rank $r \geq 2$, one has
$$
H^2(\IZ^r, \IC^{\times}) = (\IC^{\times})^{r(r-1)/2} \, .
$$
Hence, there are infinitely many isomorphism classes of $\IC[\IZ^r]$-Galois objects.
\end{exa}

Since $H=\IC[G]$ is cocommutative, we know from above that the corresponding bialgebroids $\C(A, H)$ are all 
isomorphic to $H$ as Hopf algebra. It is instructive to show this directly. Clearly, for any $u_{g}\ot u_h \in \C(A, H)$
the coinvariance condition $u_{g} \ot u_h \ot g h = u_{g} \ot u_h \ot 1_H$, requires $h=g^{-1}$ so that $\C(A, H)$ is generated as vector space by elements $u_{g}\ot u_{g^{-1}}$, $g\in G$, with multiplication 
\beq\label{prog0}
(u_{g}\ot u_{g^{-1}})\bullet_{\C}(u_{h}\ot u_{h^{-1}}) = \lambda(g, h) \lambda(h^{-1}, g^{-1}) u_{gh}\ot u_{(gh)^{-1}},
\eeq
\begin{lem}
The cocycle $\Lambda(g, h) = \lambda(g, h) \lambda(h^{-1}, g^{-1})$ is trivial in $H^2(G, \IC^{\times})$.
\end{lem}
\begin{proof}
Firstly, we can always rescale $u_e$ to $\lambda(e, e) \, u_e$ so to have $\lambda(e, e) = 1$. Then the cocycle condition \eqref{two-coc} yields $\lambda(g, e)=\lambda(e, g)=\lambda(e, e)=1$, for any $g\in G$. Next, with $u_{g}u_{h}=\lambda(g, h)u_{gh}$  and $u_{h^{-1}} u_{g^{-1}}=\lambda(h^{-1}, g^{-1}) u_{(gh)^{-1}}$, on the one hand we have 
$u_{g}u_{h} u_{h^{-1}} u_{g^{-1}} = \lambda(g, h) \lambda(h^{-1}, g^{-1}) \lambda(gh, (gh)^{-1}) u_e$. 
On the other hand $u_{g}u_{h} u_{h^{-1}} u_{g^{-1}} =  \lambda(g, g^{-1}) \lambda(h, h^{-1})u_e$.  Thus
$$
\Lambda(g, h) = \lambda(g, g^{-1}) \lambda(h, h^{-1}) / \lambda(gh, (gh)^{-1})
$$
showing $\Lambda(g, h)$ is trivial since $\Lambda(g, h) = \mu(g) \mu(h) (\mu(gh))^{-1}$ with 
$\mu(g)=\lambda(g, g^{-1})$.
\end{proof}
\noindent
Consequently, by rescaling the generators $u_{g} \to v_g = \lambda(g, g^{-1})^{-\frac{1}{2}} \, u_{g}$
the multiplication rule \eqref{mult} becomes 
$v_{g}v_{h}=\lambda'(g, h) \, v_{gh}$,
with $\lambda'(g, h)= \Lambda(g, h)^{-\frac{1}{2}} \, \lambda(g, h)$ that we rename back to $\lambda(g, h)$ from now on.
As for the bialgebroid product in \eqref{prog0} one has,
\beq\label{prog}
(v_{g}\ot v_{g^{-1}}) \bullet_{\C}(v_{h}\ot v_{h^{-1}}) = v_{gh}\ot v_{(gh)^{-1}},
\eeq
and the isomorphism $\Phi^{-1} : H \to \C(A, H) $ is simply $\Phi^{-1}(g)=\tau_L(g) = v_{g}\ot v_{g^{-1}}$.

As in \eqref{isogenH}, the group of bisections $\B(\C(A, H))$ of $\C(A, H)$, and the gauge group $\Aut_H(A)$ of the Galois object $A$ coincide with the group of characters on $\IC[G]$, which is in turn the same as $\Hom(G, \IC^{\times})$  the group (for point-wise multiplication) of group morphisms from  $G$ to $\IC^{\times}$.

Explicitly, since $F\in\Aut_H(A)$ is linear on $A$, on a basis  $\{v_{g}\}_{g\in G}$ of $A$, it is of the form
\begin{align*}
F(v_{g})=\sum_{h\in G}f_{h}(g) v_{h},
\end{align*}
for complex numbers, $f_{h}(g) \in \IC$. Then, the $H$-equivariance of $F$, 
\begin{align*}
\zero{F(v_{g})}\ot\one{F(v_{g})}=F(\zero{v_{g}})\ot \one{v_{g}}=F(v_{g})\ot g,
\end{align*}
requires $F(v_{g})$ belongs to $A_{g}$ and we get $f_{h}(g) = 0$, if $h\not=g$ while $f_{g} := f_{g}(g) \in \IC^{\times}$ from the invertibility of $F$. Finally $F$ is an algebra map:
\begin{align*}
\lambda(g, h) f_{gh} v_{gh}=F(\lambda(g, h) \, v_{gh})=F(v_{g} v_{h}) 
= F(v_{g}) F(v_{h})=\lambda(g, h) f_{g} f_{h} \, v_{gh} \, ,
\end{align*}
implies $f_{gh}=f_{g}f_{h}$, for any $g, h\in G$. Thus we re-obtain that $\Aut_{H}(A)\simeq \Hom(G, \IC^{\times})$. 
Note that the requirement $F(v_e=1_A) = 1 = F_e$ implies that $F_{g^{-1}}=(F_{g})^{-1}$.

On the other hand, the group $\Aut^{ext}_{H}(A)$ and then $\B^{ext}(\C(A, H))$ can be quite big.
If $F\in \Aut^{ext}_{H}(A)$ that is one does not require $F$ to be an algebra map, the corresponding $f_g$ 
can take any value in $\IC^{\times}$ with the only condition that $f_e = 1$. 

\subsection{Taft algebras}\label{taftal}
Let $N\geq 2$ be an integer and let $q$ be a primitive $N$-th roots of unity. 
The \emph{Taft algebra} $T_{N}$ is the $N^{2}$-dimensional unital algebra generated 
by generators $x$, $g$ subject to relations:
$$
x^{N}=0\, , \quad g^{N}=1\, , \quad xg - q\, gx=0 \, .
$$
It is a Hopf algebra with coproduct: 
\begin{align*}
\Delta(x):=1\ot x+x\ot g, \qquad \Delta(g):=g \ot g \, ;
\end{align*}
counit:
\begin{align*}
\epsilon(x):=0, \qquad \epsilon(g):=1 \, ;
\end{align*}
and antipode: 
\begin{align*}
S(x): = -g^{-1}x, \qquad S(g) := g^{-1} \, .
\end{align*}
%The order of the antipode is $2N$.
This Hopf algebra is neither commutative nor cocommutative. The four dimensional algebra $T_2$ is also known as the \emph{Sweedler algebra}.

For any $s\in \IC$, let $A_{s}$ be the unital algebra generated by elements $X, G$ with relations: 
$$
X^{N}=s\, , \quad G^{N} = 1\, , \quad X G - q \, GX=0\, . 
$$
The algebra $A_{s}$ is a right $T_{N}$-comodule algebra, with coaction defined by
\begin{align}\label{coTaft}
\delta^A(X):=1\ot x+ X \ot g,  \qquad \delta^A(G):=G \ot g.
\end{align}
Clearly, the corresponding coinvariants are just the ground field $\IC$ and so $A_{s}$ is a $T_{N}$-Galois object.
It is known (cf. \cite{Masuoka}, Prop. 2.17, Prop. 2.22, as well as \cite{schau1}) that any $T_{N}$-Galois object is isomorphic to $A_s$ for some $s\in \IC$ 
and that any two such Galois objects $A_s$ and $A_t$ are isomorphic if and only if $s=t$. 
Thus the equivalence classes of  $T_{N}$-Galois objects are in bijective correspondence with the abelian group $\IC$. 
For the corresponding Ehresmann--Schauenburg bialgebroid $\C(A_s, T_N) = (A_s\ot A_s)^{co \, T_N}$.
\begin{lem}\label{trTaft}
The translation map of the coaction \eqref{coTaft} is given on generators by
\begin{align*}
\tau(g) &= G^{-1} \ot G , \nn  \\
\tau(x) &= 1 \ot X - X G^{-1} \ot G .
\end{align*}

\begin{proof}
We just apply the corresponding canonical map to obtain:
\begin{align*}
\chi \circ \tau (g) & = G^{-1} G \ot g = 1 \ot g , \\
\chi \circ \tau (x) & = 1 \ot x+ X \ot g - XG^{-1} G \ot g = 1 \ot x .
\end{align*}
as it should be.
\end{proof}
\end{lem}
\noindent
We have then the following:
\begin{prop}\label{bitaft}
For any complex number $s$ there is a Hopf algebra isomorphism
$$
\Phi : \C(A_s, T_N) \simeq T_N .
$$
\end{prop}
\begin{proof}
It is easy to see that the elements 
\beq%\label{coTaft-b}
\Xi = X \ot G^{-1} - 1 \ot X G^{-1}  , \qquad  \Gamma = G \ot G^{-1}
\eeq
are coinvariants for the right diagonal coaction of $T_N$ on $A_s\ot A_s$ and that they generate 
$\C(A_s, T_N)=(A_s\ot A_s)^{co \, T_N}$ as an algebra. These elements satisfy the relations:
\begin{align}
\Xi^N = 0, \quad \Gamma^N = 1, \quad \Xi \bullet_{\C} \Gamma = q \, \Xi \bullet_{\C} \Gamma  \, .
\end{align}
Indeed, the last two relations are easy to see. As for the first one, shifting powers of $G^{-1}$ to the right one finds
\begin{align*}
\Xi^N & = X^N \ot G^{-N} + \sum_{r=1}^{N-1} c_r \, X^{N-r} \ot X^{r} G^{-N} + (-1)^N \ot (X G^{-1})^{N} \\
& = \big[ X^N \ot 1 + \sum_{r=1}^{N-1} c_r \, X^{N-r} \ot X^{r} + (-1)^N q^{\frac{n(n-1)}{2}}\ot X^{N} \big] \, G^{-N}
\end{align*}
for explicit coefficients $c_r$ depending on $q$. Then, using the same methods as in \cite{Taft} one shows that, 
being $q$ a primitive $N$-th roots of unity, all coefficients $c_r$ vanish and so
$\Xi^N = X^N \ot G^{-N} + (-1)^N \ot (X G^{-1})^{N}$ which then vanishes from $X^N=0$.

Thus the elements $\Xi$ and $\Gamma$ generate a copy of the algebra $T_N$ and the isomorphism 
$\Phi$ maps $\Xi$ to $x$ and $\Gamma$ to $g$. 
The map $\Phi$ is also a coalgebra map. Indeed,
$$
\Delta(\Phi(\Gamma)) = \Delta(g) =g \ot g,
$$
while, using Lemma \ref{trTaft},
$$
\Delta_\C(\Gamma) = \zero{G} \ot \tuno{\one{G}} \ot \tdue{\one{G}} \ot G^{-1} = G \ot G^{-1} \ot G \ot G^{-1} 
= \Gamma \ot \Gamma.
$$
Thus $(\Phi \ot \Phi)(\Delta_\C(\Gamma)) = g \ot g = \Delta(\Phi(\Gamma))$. Similarly, 
$$
\Delta(\Phi(\Xi)) = \Delta(x) = 1 \ot x + x \ot g,
$$
while, using Lemma \ref{trTaft} in the third step,
\begin{align*}
\Delta_\C(\Xi) & = \Delta_\C(X \ot G^{-1}) - \Delta_\C(1 \ot X G^{-1}) \\
& = \zero{X} \ot \tuno{\one{X}} \ot \tdue{\one{X}} \ot G^{-1} - 1 \ot 1 \ot 1 \ot X G^{-1} \\
& = 1 \ot \tuno{x} \ot \tdue{x} \ot G^{-1} + X \ot \tuno{g} \ot \tdue{g} \ot G^{-1}
- 1 \ot 1 \ot 1 \ot X G^{-1} \\
& = 1 \ot \Big(1 \ot X - X G^{-1} \ot G \Big) \ot G^{-1} + X \ot G^{-1} \ot G \ot G^{-1} - 1 \ot 1 \ot 1 \ot X G^{-1} \\
& = 1 \ot 1 \ot \Big( X \ot G^{-1} - 1 \ot X G^{-1} \Big) + \Big( X \ot G^{-1}  - 1 \ot X G^{-1} \Big) \ot G \ot G^{-1} \\
& = 1\ot \Xi + \Xi \ot \Gamma .
\end{align*}
  Thus $(\Phi \ot \Phi)(\Delta_\C(\Xi)) = 1 \ot x + x \ot g = \Delta(\Phi(\Xi))$. 
Finally: $\epsilon_\C(\Gamma) = 1 = \epsilon(g)$ and  $\epsilon_\C(\Xi) = 0 = \epsilon(x)$. 
This concludes the proof.
\end{proof}
The group of characters of the Taft algebra $T_N$ is the cyclic group $\IZ_N$: indeed any character $\phi$ must be such that  $\phi(x)=0$, while $\phi(g)^N=\phi(g^N)=\phi(1)=1$. Then for the group of gauge transformations of the Galois object $A_s$, the same as the group of bisections of the bialgebroid $\C(A_s, T_N)$, due to Proposition \ref{bitaft} we have,
\beq
\Aut_{T_N}(A_s) \simeq \B(\C(A_s, T_N)) = \mbox{Char}(T_N) = \IZ_N. 
\eeq

On the other hand, elements $F$ of $\Aut^{ext}_{T_N}(A_s) \simeq \B^{ext}(\C(A_s, T_N)$, due to equivariance 
$\zero{F(a)}\ot\one{F(a)}=F(\zero{a})\ot \one{a}$ for any $a \in A_{s}$, can be given as a block diagonal matrix 
$$
F = \diag(M_1, M_2, \dots, M_{N-1}, M_N)
$$
with each $M_j$ a $N \times N$ invertible lower triangular matrix

$$
M_j = \begin{bmatrix}
1  & 0    & 0      &  \dots      & 0   & 0 \\
b_{21}  & a_{N-1} & 0 & \dots &   0  & 0 \\
b_{31}  & b_{32} & a_{N-2} & \ddots & \ddots &  \vdots  \\
  \vdots & \ddots & \ddots & \ddots & 0 & 0  \\
b_{N-1, 1} &  b_{N-1, 2}      & \ddots & \ddots & a_2 & 0 \\
b_{N 1} & b_{N 2}&  \dots      & b_{N, N-2}  & b_{N, N-1}      & a_1
\end{bmatrix}
$$

\noindent
All matrices $M_j$ have in common the diagonal elements $a_j$ (ciclic permuted) which are all different from zero for the invertibility of $M_j$. For the subgroup $\Aut_{T_N}(A_s)$ the $M_j$ are diagonal as well with 
$a_k=(a_1)^k$ and $(a_1)^N=1$ so that $M_j \in\IZ_N$. 
The reason all $M_j$ share the same diagonal elements (up to permutation) is the following: firstly, the `diagonal' form of the coaction of $G$ in \eqref{coTaft} imply that 
the image $F(G^k)$ is proportional to $G^k$, say $F(G^k)=\alpha_k G^k$ for some  constant $\alpha_k$. 
Then, do to the first term in the coaction of $X$ in \eqref{coTaft}, the `diagonal' component along the basis element $X^l G^k$ of the image $F(X^l G^k)$ is given again by $\alpha_k$ for any possible value of the index $l$.

Let us illustrate the construction for the cases of $N=2, 3$. Firstly, $F(1)=1$ since $F$ is unital. 
When $N=2$, on the basis $\{1, X, G, XG\}$, the equivariance $\zero{F(a)}\ot\one{F(a)}=F(\zero{a})\ot \one{a}$ for the coaction \eqref{coTaft} becomes
\begin{align*}
\zero{F(X)}\ot\one{F(X)}&=1 \ot x+F(X)\ot g,\\
\zero{F(G)}\ot \one{F(G)}&=F(G)\ot g,\\
\zero{F(XG)}\ot\one{F(XG)}&=F(G)\ot xg+F(XG)\ot 1.
\end{align*} 
Next, write $F(a)=f_{1}(a) + f_{2}(a)\, X+f_{3}(a)\, G + f_{4}(a)\, XG$, for complex numbers $f_{k}(a)$. 
And compute $\zero{F(a)}\ot \one{F(a)}=f_{1}(a)\, 1 \ot 1+ f_{2}(a)\, (1\ot x+ X\ot g)+f_{3}(a)\, G\ot g
+f_{4}(a)\, (G\ot xg+XG\ot 1)$. Then comparing generators, the equivariance gives
\begin{align*}
& f_1(X) = f_4(X) = 0 \\
& f_1(G) = f_2(G) = f_4(G)= 0 \\
& f_2(XG) = f_3(XG) = 0 ,
\end{align*}
while the remaining coefficients are related by the system of equations
\begin{align*}
f_{2}(X)\, (1\ot x+ X\ot g)+f_{3}(X)\, G\ot g &=1 \ot x+F(X)\ot g,\\
f_{3}(G)\, G\ot g &=F(G)\ot g,\\
f_{1}(XG)\, 1 \ot 1 + f_{4}(XG)\, (G\ot xg+XG\ot 1) &=F(G)\ot xg+F(XG)\ot 1.
\end{align*} 
One readily finds solutions
\begin{align*}
& f_2(X) = 1 , \qquad f_3(X) = \gamma , \qquad f_1(XG) = \beta \\
& f_3(G) = f_4(XG) = \alpha 
\end{align*}
with $\alpha, \beta, \gamma$ arbitrary complex numbers. Thus a generic element $F$ of $\Aut^{ext}_{T_2}(A_s)$
can be represented by the matrix:
\begin{align}\label{autverT}
F : 
\begin{pmatrix}
1 \\ X G\\ G \\ X \\
\end{pmatrix} 
\quad \mapsto \quad
 \begin{pmatrix}
1 & 0 & 0 & 0\\
 \beta & \alpha & 0 & 0\\
0 & 0 & \alpha & 0\\
0 & 0 & \gamma & 1\\ 
\end{pmatrix} 
\begin{pmatrix}
1 \\ X G\\ G \\ X \\
\end{pmatrix} .
\end{align}
Asking $F$ to be invertible requires $\alpha \neq 0$.
On the other hand, any $F\in\Aut_{T_{2}}(A_{s})$ is an algebra map 
and so is determined by its values on the generators $G, X$.
From $F(G)= \alpha G$ and $F(X)= \gamma G + X$: 
requiring $s=F(X^2) = (\gamma G + X)^2 = \gamma + (GX + XG) + s$ yields $\gamma=0$; 
then $\beta + \alpha X G = F(XG) = \alpha XG$ yields $\beta=0$; and $1 = F(G^2) = (\alpha G)^2$ leads to 
$\alpha^2=1$. Thus $F(X)=X$ and $F(G)= \alpha G$, with $\alpha^2=1$ and we conclude that 
$\Aut_{T_{2}}(A_{s})\simeq\mathbb{Z}_{2}$.

When $N=3$, a similar, if longer computation, gives for $\Aut^{exp}_{T_{3}}(A_{s})$ an eight parameter group with its elements $F$ of the form   
 \begin{align}
F :
\begin{pmatrix}
1 \\ X G^{2} \\ X^{2} G \\ G^{2} \\ X G \\ X^{2} \\ G \\ X \\ X^{2} G^{-1}\\
\end{pmatrix} 
\mapsto
 \begin{pmatrix}
1 & 0 & 0 & 0 & 0 & 0 & 0 & 0 & 0\\
\beta & \alpha_2 & 0 & 0 & 0 & 0 & 0 & 0 & 0\\
\eta & - q \delta & \alpha_1 & 0 & 0 & 0 & 0 & 0 & 0\\
0 & 0 & 0 & \alpha_2 & 0 & 0 & 0 & 0 & 0\\
0 & 0 & 0 & \delta & \alpha_1 & 0 & 0 & 0 & 0\\
0 & 0 & 0 & \lambda & - q \gamma & 1 & 0 & 0 & 0\\
0 & 0 & 0 & 0 & 0 & 0 & \alpha_1 & 0 & 0\\
0 & 0 & 0 & 0 & 0 & 0 & \gamma & 1 & 0\\
0 & 0 & 0 & 0 & 0 & 0 & \theta &- q\beta & \alpha_2 \\
\end{pmatrix}  
\begin{pmatrix}
1 \\ X G^{2} \\ X^{2} G \\ G^{2} \\ X G \\ X^{2} \\ G \\ X \\ X^{2} G^{-1}\\
\end{pmatrix} .
\end{align}
One needs $\alpha_j \neq 0$, $j=1,2$ for invertibility. 
By going as before, for any $F\in\Aut_{T_{N}}(A_{s})$ one starts from it values on the generators, 
$F(G)= \alpha_1 G$ and $F(X)= \gamma G + X$, to conclude that $F$ is a diagonal matrix (in particular $F(X)=X$)
with $\alpha_2=(\alpha_1)^2$ and $1=(\alpha_1)^3$; thus $\Aut_{T_{3}}(A_{s})\simeq\mathbb{Z}_{3}$. 

\section{Crossed module structures on bialgebroids} %\label{cmhg}
Isomorphisms of (a usual) groupoid with natural transformations between them form a strict 2-groupoid. In particular, 
automorphisms of the groupoid  with its natural transformations, form a strict 2-group or, equivalently, a crossed module (cf. \cite{ralf-chen}, Definition 3.21). The crossed module combines automorphisms of the groupoid and bisections since the latter are the natural transformations from the identity functor to  automorphisms. The crossed module involves the product on bisections and the composition on automorphisms, and the group homomorphism from bisections to automorphisms together with the action of automorphisms on bisections by conjugation. Any bisection $\sigma$ is the 2-arrow from the identity morphism to an automorphism $Ad_\sigma$, and the composition of bisections can be viewed as the horizontal composition of 2-arrows. 

In this section we quantise this construction for the Ehresmann--Schauenburg bialgebroid of a Hopf--Galois extension. We construct a crossed module for the bisections and the automorphisms of the bialgebroid. Notice that we do not  need the antipode of bialgebroid, that is we do not need to defined the crossed module on Hopf algebroid and the crossed module on bialgebroid is a generalization of the crossed module on groupoid. In the next section, 
The construction can also be repeated for extended bisections.

\subsection{Automorphisms and crossed modules} \label{cmhg}
Recall that a crossed module is the data $(M, N, \mu, \alpha)$ of two groups $M$, $N$ together with 
a group morphism $\mu: M\to N$ and a group morphism $\alpha : N \to \Aut(M)$ such that, 
denoting $\alpha_{n} : M\to M$ for every $n\in N$, the following conditions are satisfied:
\begin{itemize}
    \item[(1)] $\mu(\alpha_{n}(m))=n\mu(m) n^{-1}$, \quad for any $n\in N$ and $m\in M$;
    ~\\
    \item[(2)] $\alpha_{\mu(m)}(m')=mm'm^{-1}$, \quad for any $m, m'\in M$.
\end{itemize}

Then, with the definition of the automorphism group of a bialgebroid as given in Definition \ref{amoeba},
we aim at proving the following.

\begin{thm}\label{thm. crossed module1}
Given a Hopf--Galois extension $B=A^{coH}\subseteq A$, let $\mathcal{C}(A, H)$ be the corresponding left Ehresmann--Schauenburg bialgebroid, and assume $B$ is in the centre of $A$. Then there is a group morphism $Ad: \B(\C(A, H))\to\Aut(\C(A, H))$ and an action $\triangleright$ of $\Aut(\C(A, H))$ on $\B(\C(A, H))$ that give a crossed module structure to 
$\big( \B(\C(A, H)), \Aut(\C(A, H)) \big)$. 
\end{thm}

We give the proof in a few lemmas.
\begin{lem}\label{ajaut}
Given a Hopf--Galois extension $B=A^{coH}\subseteq A$, let $\mathcal{C}(A, H)$ be the corresponding left Ehresmann--Schauenburg bialgebroid. Assume $B$ belongs to the centre of $A$. 
For any bisection $\sigma\in \B(\C(A, H))$, denote $ad_\sigma = \sigma \circ s \in \Aut(B)$ and let
$F_{\sigma}$ be the associated gauge element in $\Aut_H(A)$ (see \eqref{btog}). 
Define 
$Ad_\sigma: \C(A, H)\to \C(A, H)$ by
\begin{align}
Ad_\sigma(a\ot\tilde{a}) & :=F_{\sigma}(a)\ot F_{\sigma}(\tilde{a}) \nn \\ & \:= 
\sigma(\zero{a}\ot\tuno{\one{a}})\tdue{\one{a}}\ot\sigma(\zero{\tilde{a}}\ot\tuno{\one{\tilde{a}}})\tdue{\one{\tilde{a}}}\, . \label{Ad}
\end{align}
Then the pair $(Ad_\sigma, ad_\sigma)$ is an automorphism of $\C(A, H)$.
\end{lem}
\begin{proof}
Since $F_{\sigma}$ is an algebra automorphism, so is $Ad_\sigma$. Then, for any $b\in B$,
\begin{align*}
Ad_\sigma(t(b)) = Ad_\sigma(1\ot b) = 1\ot\sigma(b\ot 1) = t(ad_\sigma(b)) 
\end{align*}
and
\begin{align*}
Ad_\sigma(s(b)) = Ad_\sigma(b\ot 1) = \sigma(b\ot 1)\ot1 = s(ad_\sigma(b)) .
\end{align*}
So conditions (i) and (ii) of Definition \ref{amoeba} are satisfied. 
For condition (iii), using $H$-equivariance of $F_{\sigma}$, with $a\ot\tilde{a}\in \C(A, H)$ we get 
\beq\label{ad1}
(\Delta_{\C(A, H)}\circ Ad_\sigma) (a\ot\tilde{a}) 
= F_{\sigma}(\zero{a})\ot \tuno{\one{a}}\ot_{B}\tdue{\one{a}}\ot F_{\sigma}(\tilde{a}). 
\eeq
On the other hand, 
\beq\label{ad2}
\big((Ad_\sigma\ot_{B} Ad_\sigma)\circ \Delta_{\C(A, H)} \big) (a\ot\tilde{a}) = F_{\sigma}(\zero{a})\ot F_{\sigma}(\tuno{\one{a}})\ot_{B}F_{\sigma}(\tdue{\one{a}})\ot F_{\sigma}(\tilde{a}).
\eeq
Now, for any $F \in \Aut_H(A)$,  given $h \in H$, one has 
\begin{align}\label{idtens}
 F(\tuno{h}) & \ot_{B}F(\tdue{h}) = \tuno{h}\ot_{B}\tdue{h} , \qquad \mbox{for any} \quad h \in H .
\end{align}
By applying the canonical map $\chi$ and using equivariance of $F$ we compute,
\begin{align*}
\chi( F(\tuno{h}) \ot_{B}F(\tdue{h})) & = F(\tuno{h}) \zero{F(\tdue{h})} \ot \one{F(\tdue{h})} \\
& = F(\tuno{h}) F(\zero{\tdue{h}}) \ot \one{\tdue{h}} \\
& = F(1_A) \ot h = 1_A \ot h
\end{align*}
using \eqref{p7}. Being $\chi$ an isomorphism we get the relation \eqref{idtens}. Using this for 
the right hand sides of \eqref{ad1} and \eqref{ad2} shows that they coincide and condition (iii) is satisfied. 
Finally, 
\begin{align*}
(\epsilon\circ Ad_\sigma) (a\ot\tilde{a}) = \epsilon(F_{\sigma}(a)\ot F_{\sigma}(\tilde{a}))=F_{\sigma}(a\tilde{a})=(\sigma \circ s)(a\tilde{a})=ad_\sigma\circ \epsilon(a\ot\tilde{a}).
\end{align*}
This finishes the proof.
\end{proof}

\begin{rem}
The map $Ad_\sigma$ in \eqref{Ad} can also be written in the following useful ways:
\begin{align}
    Ad_\sigma(a\ot \tilde{a})& = \sigma(\zero{a}\ot\tuno{\one{a}})\tdue{\one{a}}\ot \tuno{\two{a}}\sigma^{-1}( \tdue{\two{a}} \ot \tilde{a}) \nn \\
   & = \sigma(\one{(a\ot \tilde{a})}) \, \two{(a\ot \tilde{a})} \, 
  (\sigma\circ s) \circ \sigma^{-1} (\three{(a\ot \tilde{a})} ) ) \label{altad}
 \end{align}
Indeed, for $a\ot a'\in \C(A, H)$, by inserting \eqref{p5} and using the definition 
of the inverse $\phi^{-1}$, we compute,
\begin{align*}
    Ad_\sigma(a\ot \tilde{a})& = \sigma(\zero{a}\ot\tuno{\one{a}})\tdue{\one{a}}\ot \sigma(\zero{\tilde{a}}\ot \tuno{\one{\tilde{a}}})
   \tdue{\one{\tilde{a}}}
    \\        & = \sigma(\zero{a}\ot\tuno{\one{a}})\tdue{\one{a}}\ot \tuno{\two{a}}   \tdue{\two{a}} 
    \sigma(\zero{\tilde{a}}\ot  \tuno{\one{\tilde{a}}})\tdue{\one{\tilde{a}}} \\ 
    & = \sigma(\zero{a}\ot\tuno{\one{a}})\tdue{\one{a}}\ot \tuno{\two{a}} \sigma(\zero{\tilde{a}}\ot  \tuno{\one{\tilde{a}}})\tdue{\two{a}} \tdue{\one{\tilde{a}}}\\
    & = \sigma(\zero{a}\ot\tuno{\one{a}})\tdue{\one{a}}\ot \tuno{\two{a}} 
   (\sigma\circ s) \circ \sigma^{-1}  
    ( \tdue{\two{a}} \ot \tilde{a}) \big)
      \\ & = \sigma(\one{(a\ot \tilde{a})}) \, \two{(a\ot \tilde{a})} \, 
      (\sigma\circ s) \circ \sigma^{-1} (\three{(a\ot \tilde{a})} )
\end{align*}
\end{rem}
\noindent
It is easy to see that $Ad_\sigma \circ Ad_\tau = Ad_{\tau \ast \sigma}$ for any
$\sigma_{1}$, $\sigma_{2}\in \B(\C(A, H))$, while $(Ad_\sigma)^{-1} = Ad_{\sigma^{-1}}$ 
and $Ad_\epsilon = id_{\C(A, H)}$. And, of course
$ad_\sigma \circ ad_\tau = ad_{\tau \ast \sigma}$, with $(ad_\sigma)^{-1} = ad_{\sigma^{-1}}$ 
and $ad_\epsilon = id_B$. Thus $Ad$ is a group morphism $Ad: \B(\C(A, H))\to\Aut(\C(A, H))$.

Next, given an automorphism $(\Phi, \varphi)$ of $\mathcal{C}(A, H)$ with inverse $(\Phi^{-1}, \varphi^{-1})$, 
we define an action of $(\Phi, \varphi)$ on the group of bisections $\B(\C(A, H))$ as follow:
\beq\label{aaobs}
\Phi \triangleright \sigma :=\varphi^{-1} \circ \sigma \circ \Phi ,
\eeq
for any $\sigma\in \B(\C(A, H))$. The result is an algebra map since it is a composition of algebra map. Moreover, for any $b\in B$, $(\Phi\triangleright\sigma)( t(b))=\varphi^{-1}(\sigma(t(\varphi(b))))=\varphi^{-1}(\varphi(b))=b$, so that $(\Phi\triangleright\sigma) \circ t = \id_B$; while 
$(\Phi\triangleright\sigma)(s(b))=\varphi^{-1}\big( (\sigma \circ s)(\varphi(b)) \big)$, so that $(\Phi\triangleright\sigma)\circ s\in\Aut(B)$. 
And one checks that
\beq\label{aaobs-1}
(\Phi\triangleright\sigma)^{-1} = \Phi\triangleright \sigma ^{-1} = \varphi^{-1} \circ \sigma^{-1} \circ \Phi .
\eeq

\begin{lem}\label{aaction}
Given any automorphism $(\Phi, \varphi)$, the action defined in \eqref{aaobs} is a group automorphism of $\B(\C(A, H))$.
\end{lem}
\begin{proof}
Let $\sigma, \tau\in \B(\C(A, H))$, and $c \in \C(A, H)$, we compute:
\begin{align*}
(\Phi\triangleright\tau)\ast(\Phi\triangleright\sigma)(c)&=(\Phi\triangleright\sigma)(s(\Phi\triangleright\tau(\one{c})))(\Phi\triangleright\sigma)(\two{c})\\
&= \big( \varphi^{-1} \circ \sigma \circ \Phi \circ s \circ \varphi^{-1} \circ \tau \circ \Phi (\one{c}) \, \big) 
\big(\varphi^{-1} \circ \sigma \circ \Phi(\two{c}) \big)\\
&= \big( \varphi^{-1} \circ \sigma \circ s \circ \tau \circ \Phi (\one{c}) \, \big) 
\big(\varphi^{-1} \circ \sigma \circ \Phi(\two{c}) \big)\\
&=\varphi^{-1} \big(\sigma \circ s \circ \tau(\Phi(\one{c})) \, \sigma (\Phi(\two{c}) ) \big)\\
&=\varphi^{-1} \circ (\tau\ast\sigma)\circ \Phi(c)\\
&=\Phi\triangleright(\tau\ast\sigma)(c),
\end{align*}
where the last but one step uses condition (iii) of Definition \ref{amoeba}. Also, 
\begin{align*}
\Phi \triangleright \epsilon = \varphi^{-1} \circ \epsilon \circ \Phi = \varphi^{-1} \circ \varphi \circ \epsilon = \epsilon.
\end{align*}
Finally, for any two automorphism $(\Phi, \varphi)$ and $(\Psi, \psi)$ of $\C(A, H)$, we have
\begin{align*}
\Phi\triangleright(\Psi\triangleright(\sigma)) = \varphi^{-1} \circ \psi^{-1} \circ \sigma \circ \Psi \circ \Phi = (\psi \varphi)^{-1} \circ 
\sigma \circ \Psi \circ \Phi
=(\Psi\circ \Phi)\triangleright\sigma .
\end{align*}
In particular $\Phi^{-1}\triangleright(\Phi\triangleright(\sigma))=\sigma$ and so the action is an automorphism 
of $\B(\C(A, H))$.
\end{proof}

\begin{lem}\label{d2}
For any automorphism $(\Phi, \varphi)$, and any $\sigma\in \B(\C(A, H))$ 
we have 
$$
Ad_{\Phi \triangleright \sigma}=\Phi^{-1} \circ Ad_\sigma \circ \Phi .
$$
\end{lem}
\begin{proof}
With $a\ot\tilde{a}\in \C(A, H)$, from \eqref{altad} we get
\begin{align}
    (Ad_\sigma \circ \Phi) (a\ot \tilde{a}) 
    = \sigma(\one{(\Phi (a\ot \tilde{a}))}) \, \two{(\Phi (a\ot \tilde{a}))} \, 
   (\sigma\circ s) \circ \sigma^{-1}  (\three{(\Phi (a\ot \tilde{a}))}) , \label{11}
 \end{align}
while, using \eqref{aaobs} and \eqref{aaobs-1}, we have
\begin{align*}
    Ad_{\Phi\triangleright\sigma} (a\ot \tilde{a}) & = (\Phi\triangleright\sigma) (\one{(a\ot \tilde{a})}) \, \two{(a\ot \tilde{a})} \,  
 \big( (\Phi\triangleright\sigma) \circ s \big) \circ (\Phi\triangleright\sigma)^{-1} 
    (\three{(a\ot \tilde{a})} ) \\
    & = \varphi^{-1} \big( \sigma (\Phi(\one{(a\ot \tilde{a})}) ) \big) \, \two{(a\ot \tilde{a})} \, 
   \big( (\Phi\triangleright\sigma) \circ s \big) \circ (\varphi^{-1} \circ \sigma^{-1} \big)  (\Phi(\three{(a\ot \tilde{a})}) ) . 
    %\label{22}
 \end{align*}
Since $\Phi$ is a bimodule map: $\Phi(b (a\ot \tilde{a}) \tilde{b}) = \varphi(b) \Phi(a\ot \tilde{a}) \varphi(\tilde{b})$, for all $b,\tilde{b} \in B$, we get, 
\begin{align}
 (\Phi \circ Ad_{\Phi\triangleright\sigma}) (a\ot \tilde{a})  
 = \sigma (\Phi(\one{(a\ot \tilde{a})}) ) \, \Phi(\two{(a\ot \tilde{a})}) \, 
   (\sigma \circ s) \circ \sigma^{-1} (\Phi(\three{(a\ot \tilde{a})}) ) .
  \label{33}
 \end{align}
That the right hand sides of \eqref{11} and \eqref{33} are equal follows from the equavariance 
condition (iii) of Definition \ref{amoeba}. 
\end{proof}

\begin{lem}\label{d3}
Let $\sigma$, $\tau\in \B(\C(A, H))$, then
$Ad_\tau \triangleright\sigma=\tau\ast\sigma\ast\tau^{-1}$.
\end{lem}
\begin{proof}
With  $a\ot\tilde{a}\in \C(A, H)$, using the definition \eqref{altad} we compute
\begin{align*}
 Ad_\tau \triangleright\sigma & (a\ot\tilde{a}) 
= (ad_\tau^{-1} \circ \sigma)  (Ad_\tau (a\ot\tilde{a}) ) \\
& = (( \tau \circ s)^{-1} \circ \sigma) \big( \tau(\one{(a\ot \tilde{a})}) \, \two{(a\ot \tilde{a})} \, 
  (\tau \circ s) \circ \tau^{-1}  (\three{(a\ot \tilde{a})}) \big)\\
& = (\tau^{-1} \circ s) \Big( (\sigma \circ s) \big(\tau(\one{(a\ot \tilde{a})}) \big)\, \sigma(\two{(a\ot \tilde{a})}) \, 
  (\sigma \circ t) \circ (\tau \circ s \circ \tau^{-1})  (\three{(a\ot \tilde{a})} ) \Big)
\\
& = (\tau^{-1} \circ s) \Big( (\sigma \circ s) \big(\tau(\one{(a\ot \tilde{a})}) \big)\, \sigma(\two{(a\ot \tilde{a})}) \, 
  (\tau \circ s) \circ \tau^{-1}  (\three{(a\ot \tilde{a})} ) \Big) \\
& = (\tau^{-1} \circ s) \Big( (\tau\ast\sigma) (\one{(a\ot \tilde{a})}) \Big) \, \tau^{-1} (\two{(a\ot \tilde{a})} )  \\
& = \tau\ast\sigma\ast\tau^{-1}(a\ot\tilde{a}),
\end{align*}
where we used $(\tau\circ s)^{-1}=\tau^{-1}\circ s$, $\sigma\circ t=\id_{B}$ and definition \eqref{mulbis1} for 
the product. 
\end{proof}
Taken together the previous lemmas establish that a crossed module structure to 
$\big( \B(\C(A, H)), \Aut(\C(A, H)), Ad, \triangleright \big)$, which is the content of Theorem \ref{thm. crossed module1}.

\subsection{CoInner authomorphisms of bialgebroids}
Given a Hopf algebra $H$ and a character $\phi: H\to \mathbb{C}$, one defines a Hopf algebra automorphisms 
(see \cite[page 3807]{schau}) by
\begin{align}\label{autohopf}
\mbox{coinn}(\phi):H\to H , \qquad \mbox{coinn}(\phi)(h):= \phi(\one{h}) \two{h} \phi(S(\three{h})),
\end{align}
for any $h\in H$. Recall that for a character $\phi^{-1} = \phi \circ S$.
The set $\mbox{CoInn}(H)$ of co-inner authomorphisms of $H$ is a normal subgroup of the group
$\Aut_{\mbox{\tiny{Hopf}}}(H)$ of Hopf algebra automorphisms (this is just  $\Aut(H)$ if one view $H$ 
as a bialgebroid over $\mathbb{C}$). 

We know from the previous sections that for a Galois object $A$ of a Hopf algebra $H$, the corresponding 
bialgebroid $\C(A, H)$ is a Hopf algebra. Also, the group of gauge transformations of the Galois object which is the same as the group of bisections can be identified with the group of characters of $\C(A, H)$ (see \eqref{isogen}). 
It turns out that these groups are also isomorphic to $\mbox{CoInn}(\C(A, H))$. 
We have the following lemma:

\begin{lem}\label{adjascoinner}
For a Galois object $A$ of a Hopf algebra $H$, let $\C(A, H)$ be the corresponding bialgebroid of $A$. 
If $\phi\in \B(\C(A, H)) = \textup{Char}(\C(A, H))$, then $Ad_\phi=\textup{coinn}(\phi)$.
\end{lem}
\begin{proof}
Let $\phi\in \mbox{Char}(\C(A, H))$; then $\phi^{-1}=\phi\circ S_{\C}$. Substituting the latter in \eqref{altad}, 
for $a\ot a'\in \C(A, H)$, we get 
\begin{align*}
    Ad_\phi (a\ot \tilde{a})& = \phi(\one{(a\ot \tilde{a})}) \, \two{(a\ot \tilde{a})} \, (\phi\circ S_{\C})(\three{(a\ot \tilde{a})} ) \\
    & = \mbox{coinn}(\phi)(a\ot \tilde{a}) , 
\end{align*}
as claimed.
\end{proof}

\begin{exa}\label{crossedmodule1}
Let us consider again the Taft algebra $T_{N}$ of Section \ref{taftal}. We know from 
Proposition \ref{bitaft} that for any $T_{N}$-Galois object $A_s$ the bialgebroid $\C(T_{N}, A_{s})$ 
is isomorphic to $T_{N}$ and bisections of $\C(T_{N}, A_{s})$ are the same as characters of $T_{N}$ the group of which is isomorphic to $\IZ_N$. A generic character is a map $\phi_{r}:T_{N}\to \mathbb{C}$, 
given on generators $x$ and $g$ by 
$\phi_{r}(x)=0$ and $\phi_{r}(g)=r$ for $r$ a $N$-root of unity $r^N=1$.  
The corresponding automorphism $Ad_{\phi_{r}}=\mbox{coinn}(\phi_{r})$ is easily found to be on generators given by 
$$
\mbox{coinn}(\phi_{r})(g) = g , \qquad \mbox{coinn}(\phi_{r})(x) = r^{-1} x \, .
$$
It is known (cf. \cite{schau1}, Lemma 2.1) that
$\Aut(T_{N}) \simeq \Aut_{\mbox{\tiny{Hopf}}}(T_{N}) \simeq \mathbb{C}^{\times}$:  Indeed, given $r\in \mathbb{C}^{\times}$, one defines an authomorphism $F_{r}: T_{N}\to T_{N}$ by $F_{r}(x):=rx$ and $F_{r}(g):=g$. Thus 
$Ad: \mbox{Char}(T_{N})\to \Aut(T_{N})$ is the injection sending $\phi_{r}$ to $F_{r^{-1}}$. 

Moreover, for $F\in \Aut(T_{N})$ and $\phi\in \mbox{Char}(T_{N})$, one checks that 
$Ad_{F\triangleright\phi}(x)=Ad_\phi (x)$ and $Ad_{F\triangleright \phi} (g )=Ad_\phi (g)$. Thus, 
as a crossed module, the action of $\Aut(T_{N})$ on $\mbox{Char}(T_{N})$ is trival and the crossed module 
$(\mbox{Char}(\C(T_{N}, A_{s})),\Aut(\C(T_{N}, A_{s})), Ad, \id)$ is  isomorphic to $(\mathbb{Z}_{N}, \mathbb{C}^{\times}, i, \id)$, with inclusion $i:\mathbb{Z}_{N}\to \mathbb{C}^{\times}$ given by $i(r):=e^{-i2r\pi/N}$ and $\mathbb{C}^{\times}$ acting trivially on $\mathbb{Z}_{N}$.
\end{exa}

\subsection{Crossed module structures on extended bisections}
In parallel with the crossed module structure on bialgebroid automorphisms and bisections, there is a similar structure  on the set of `extended' bialgebroid automorphisms and extended bisections. 

Given a left bialgebroid $(\cL, \Delta, \epsilon,s,t)$ be a left bialgebroid over the algebra $B$. 
An \textup{extended automorphism} of $\cL$ is a pair $(\Phi, \varphi)$ with $\varphi : B \to B$ an algebra map
and a unital invertible linear map $\Phi: \cL \to \cL$, 
obeying the properties $(i)-(iv)$ of Definition \ref{amoeba}

So, an extended automorphism is not required in general to be an algebra map while still satisfying all other properties 
of  an automorphism. In particular we still have the bimodule property: $\Phi(b a \tilde{b}) = \varphi(b) \Phi(a) \varphi(\tilde{b})$. We denote by $\Aut^{ext}(\cL)$ the group (by composition) of extended automorphisms of $\cL$.
There is an analogous of Theorem \ref{thm. crossed module1}:
\begin{thm}\label{thm. crossed module2}
Given a Hopf--Galois extension $B=A^{coH}\subseteq A$, let $\mathcal{C}(A, H)$ be the corresponding left Ehresmann--Schauenburg bialgebroid, and assume $B$ be in the centre of $A$. Then there is a group morphism 
$Ad: \B^{ext}(\C(A, H))\to \Aut^{ext}(\C(A, H))$ and an action 
$\triangleright$ of $\Aut^{ext}(\C(A, H))$ on $\B^{ext}(\C(A, H))$ such that the group of extended automorphisms 
$\Aut^{ext}(\C(A, H))$ and of extended bisections $\B^{ext}(\C(A, H))$, form a crossed module.
\end{thm}
This result is established in parallel and similarly to the proof of Theorem \ref{thm. crossed module1}. Here we shall only point to the differences in the definitions and the proofs.

Thus, under the hypothesis of Theorem \ref{thm. crossed module2}, 
for any bisection $\sigma\in \B^{ext}(\C(A, H))$, define 
$Ad_\sigma: \C(A, H)\to \C(A, H)$, for any $a\ot \tilde{a} \in \C(A, H)$, by
\begin{align}\label{Ad1}
Ad_\sigma(a\ot \tilde{a}) &:=(\sigma(\zero{a}\ot\tuno{\one{a}})\tdue{\one{a}})\ot
  \big(\tuno{\two{a}}(\sigma\circ s) \circ \sigma^{-1}(\tdue{\two{a}}\ot \tilde{a}) \big)  , \nn \\
&\:=\sigma(\one{(a\ot \tilde{a})}) \, \two{(a\ot \tilde{a})} \, 
 (\sigma\circ s) \circ \sigma^{-1}(\three{(a\ot \tilde{a})})  
\end{align}
in parallel with \eqref{altad}. 
Then the pair of map $(Ad_\sigma, ad_\sigma = \sigma \circ s)$ is an extended automorphism of $\C(A, H)$.
In particular we have, for any $c = a \ot \tilde{a} \in \C(A, H)$, 
\begin{align*}
    \Delta_{\C}(Ad_\sigma(c))&=\sigma(\one{c}) \, \two{c} \ot_{B} \three{c} \, 
     (\sigma\circ s) \circ \sigma^{-1}(\four{c})  
    \\
    &=\sigma(\one{c}) \, \two{c} \,    (\sigma\circ s) \circ \sigma^{-1}(\three{c})   \ot_{B} 
    \sigma(\four{c}) \, \five{c} \,   (\sigma\circ s) \circ \sigma^{-1}(\six{c})   \end{align*}
where the 2nd step use $(\sigma\circ s) \circ \sigma^{-1}(\one{c}) \, \sigma(\two{c})=\epsilon(c)$. 
With the latter, we have also,
\begin{align*}
    (\epsilon\circ Ad_\sigma)(c)&=\sigma(\one{c}) \, \epsilon(\two{c}) \, 
     (\sigma\circ s) \circ \sigma^{-1}(\three{c})   \\
    &=\sigma(\one{c}) \,  (\sigma\circ s) \circ \sigma^{-1}(\two{c})   \\
    &=(\sigma\circ s) \, \epsilon(c)\\
    &=ad_\sigma(\epsilon(c)).
\end{align*}
Moreover, for two extended bisections $\sigma$ and $\tau$ we have, for $c\in \C(A, H)$,
\begin{align*}
   Ad_\sigma &\circ Ad_\tau (c)
    =Ad_\sigma \big(\tau(\one{c}) \, \two{c} \,   (\tau\circ s) \circ \tau^{-1}(\three{c})   \big)\\
    &= s \big(ad_\sigma (\tau(\one{c})) \big) \, Ad_\sigma(\two{c}) \, t \,  \big(ad_\sigma \circ (\tau \circ s) \circ \tau^{-1}(\three{c}) \big) \\
    &=s\Big(\sigma\big(s(\tau(\one{c}))\big)\Big)\Big(\sigma(\two{c}) \, \three{c}\, \big(\sigma(s(\sigma^{-1}(\four{c})))\big)\Big) \,   t \circ \big( \sigma\circ s) \circ (\tau \circ s) \circ \tau^{-1}(\five{c}) \big) \\
    &=\Big(\sigma\big(s(\tau(\one{c}))\big)\sigma\big(\two{c}\big)\Big) \, \three{c} \, \Big((\sigma\circ s\circ \tau\circ s) \big((\tau^{-1}\circ s) \circ \sigma^{-1}(\four{c}) \, \tau^{-1}(\five{c})\big)   \Big)\\
    &=(\tau\ast \sigma)(\one{c}) \, \two{c} \, \big((\tau\ast\sigma)\circ s)(\sigma^{-1}\ast\tau^{-1}(\three{c})\big)\\
    &=Ad_{\tau\ast \sigma}(c),
\end{align*}
with the 2nd step using $Ad_\sigma$ is a $B$-bimodule map. 
One also shows $ad_\sigma \circ ad_\tau = ad_{\tau \ast \sigma}$ 
and $(Ad_\epsilon, ad_\epsilon) = (\id_{\C(A, H)}, \id_{B})$. 
Therefore $(Ad_\sigma, ad_\sigma)$ is invertible with inverse $(Ad_{\sigma^{-1}}, ad_{\sigma^{-1}})$.

When $\sigma$ is an algebra map, \eqref{Ad1} reduces to \eqref{Ad} (or equivalently to\eqref{altad}). 

If $(\Phi, \varphi)$ is an extended automorphism of $\mathcal{C}(A, H)$ with inverse $(\Phi^{-1}, \varphi^{-1})$
the formula \eqref{aaobs} is an action of $(\Phi, \varphi)$ on $\B^{ext}(\C(A, H))$, a group 
automorphism of $\B^{ext}(\C(A, H))$.
 
We only check $F\triangleright\sigma$ is well defined as an extended bisection 
since the rest goes as in the previous section. For $a\ot \tilde{a}\in \C(A, H)$ and $b\in B$, 
a direct computation yields:
\begin{align*}
(\Phi \triangleright \sigma)((a\ot \tilde{a}) \triangleleft b) & = (\Phi \triangleright \sigma)(a\ot \tilde{a}) \, b \\
(\Phi \triangleright\sigma)(b\triangleright(a\ot \tilde{a})) & = (\Phi \triangleright\sigma)(s(b)) \, 
(\Phi\triangleright\sigma)(a\ot \tilde{a}).
\end{align*}
Finally, with similar computation as those of Lemma \ref{d2} and Lemma \ref{d3} one shows that
for any extended automorphism $(\Phi, \varphi)$, and any $\sigma\in \B^{ext}(\C(A, H))$ one has  
$$
Ad_{\Phi \triangleright\sigma} = \Phi^{-1}\circ Ad_\sigma \circ \Phi .
$$
And that, with $\sigma$, $\tau\in \B^{ext}(\C(A, H))$, one has
$$
Ad_\tau \triangleright\sigma=\tau\ast\sigma\ast\tau^{-1} .
$$

\begin{exa}
Consider a $H$-Galois object $A$ and let $\C(A, H)$ be the corresponding bialgebroid, a Hopf algebra itself. 
Given an extended bisection $\sigma \in \B^{ext}(\C(A, H)) \simeq \mbox{Char}^{ext}(H)$, the expression \eqref{Ad1} reduces to
$$  
Ad_\sigma(a\ot \tilde{a}) = \sigma(\one{(a\ot \tilde{a})}) \, \two{(a\ot \tilde{a})} \, \sigma^{-1}(\three{(a\ot \tilde{a})}) 
$$
In analogy with \eqref{autohopf} to which it reduces when $\phi$ is a character (and Lemma \ref{adjascoinner}) 
we may thing of this unital invertible coalgebra map as defining an extended coinner authomorphism of $\C(A, H)$,
$\mbox{coinn}(\sigma)(c):= Ad_\sigma(c) = \sigma(\one{c}) \two{c}\sigma^{-1}(\three{c})$.
\end{exa}

In Example \ref{crossedmodule1} we constructed an Abelian crossed module for the Taft algebras. The following example present a non-Abelian crossed module for Taft algebras with respect to the extended characters and 
extended automorphisms.

\begin{exa}
We know from Section \ref{taftal} that the Schauenburg bialgebroid $\C(A_{s}, T_{N})$ of any Galois object $A_s$ for the Taft algebra $T_{N}$, is isomorphic to $T_{N}$ itself. Thus $\Aut^{ext}(\C(A_{s}, T_{N})) \simeq 
\Aut^{ext}(T_{N})$ is the group of unital invertible coalgebra maps: maps
$\Phi: T_{N} \to T_{N}$ such that $\Phi(\one{h})\ot \Phi(\two{h})=\one{\Phi(h)}\ot\two{\Phi(h)}$ for any $h\in T_{N}$ with 
$\Phi(1)=1$. 

Let us illustrate this for the case $N=2$. The coproduct of $T_{2}$ on the generators $x,g$ will then require the following condition for an automorphism $\Phi$:
\begin{align}
\one{\Phi(g)}\ot\two{\Phi(g)} &= \Phi(g) \ot \Phi(g) \nn \\
\one{\Phi(x)}\ot\two{\Phi(x)} &= 1 \ot \Phi(x) + \Phi(x) \ot g \nn \\
\one{\Phi(xg)}\ot\two{\Phi(xg)} &= g \ot \Phi(xg) + \Phi(xg) \ot 1 \, . 
\end{align}
A little algebra then shows that
\begin{align}
\Phi(g) & = g  \nn \\
\Phi(x) & =  c\, ( g - 1 ) + a_2 \, x \nn \\
\Phi(xg) & = b\, ( 1 - g ) + a_1 \, xg
\end{align}
for arbitrary parameters $b, c \in \IC$ and $a_1, a_2 \in \IC^{\times}$ (for $\Phi$ to be invertible). 
As in \eqref{autverT} we can represent $\Phi$ as a matrix:
\begin{align}\label{authopfT}
\Phi : \begin{pmatrix}
1\\ xg\\ g\\ x\\
\end{pmatrix} 
\quad \mapsto \quad
 \begin{pmatrix}
1 & 0 & 0 & 0\\
b & a_1 & -b & 0\\
0 & 0 & 1 & 0\\
-c & 0 & c & a_2\\
\end{pmatrix}
\begin{pmatrix}
1\\ xg\\ g\\ x\\
\end{pmatrix} .
\end{align}
One checks that matrices $M_\Phi                  $ of the form above form a group: 
$\Aut^{ext}(T_{N}) \simeq \Aut_{\mbox{\tiny{Hopf}}}(T_{N})$. 

Given $\sigma \in \mbox{Char}^{ext}(T_{2})$ we shall denote $\sigma_a = \sigma(a) \in \IC$ for $a\in \{1, x, g, xg\}$. For the convolution inverse  $\sigma^{-1}$, from the condition $\sigma\ast\sigma^{-1}=\epsilon$ we get on the basis that 
\begin{equation}\label{extended character}
\left \{
  \begin{aligned}
    &\sigma_1 = (\sigma^{-1})_1 = 1, \\
    &\sigma_{g}(\sigma^{-1})_{g}=1,  \\
    &\sigma_{g}(\sigma^{-1})_{x}+\sigma_{x}=0,\\
    &\sigma_{g}(\sigma^{-1})_{xg}+\sigma_{xg}=0 
  \end{aligned} \right.
\end{equation}
from which we solve
\begin{equation}\label{ext-cha2}
\left \{
  \begin{aligned}
    &(\sigma^{-1})_{g} = (\sigma_{g})^{-1},  \\
    &(\sigma^{-1})_{x} = -\sigma_{x} (\sigma_{g})^{-1} ,\\
    &(\sigma^{-1})_{xg} = -\sigma_{xg} (\sigma_{g})^{-1}  
  \end{aligned} \right.
\end{equation}
Then computing $Ad_\sigma(h) = \sigma(\one{h}) \two{h} \sigma^{-1}(\three{h})$ leads to
\begin{align}
\label{ad-taft}
Ad_\sigma \begin{pmatrix}
1\\ xg\\ g\\ x\\
\end{pmatrix} 
=
\begin{pmatrix}
1 & 0 & 0 & 0\\
\sigma_{xg} & \sigma_{g} & -\sigma_{xg} & 0\\
0 & 0 & 1 & 0\\
- \sigma_{x} (\sigma_{g})^{-1} & 0 & \sigma_{x} (\sigma_{g})^{-1} & (\sigma_{g})^{-1}\\
\end{pmatrix}.
\end{align}
We see that the matrix \eqref{ad-taft} is of the form \eqref{authopfT} with the restriction that $a_2=a_1^{-1}$ so that $Ad_\phi$ has determinant $1$. Clearly, the image of $\mbox{Char}^{ext}(T_{2})$ form a subgroup of 
$\Aut^{ext}(\C(A_{s}, T_{2})) \simeq \Aut^{ext}(T_{N})$.
Moreover, $Ad: \mbox{Char}^{ext}(T_{2})\to \Aut^{ext}(T_{2})$ is an injective map. Finally, the action $Ad_{\Phi\triangleright\sigma}$ will have as matrix just the product:
\begin{align}
& M_{Ad_{\Phi\triangleright\sigma}} = M_{\Phi^{-1}} \, M_{Ad_{\sigma}} \, M_\Phi \\
\qquad & =
 \begin{pmatrix}
1 & 0 & 0 & 0\\
a_1^{-1} [ \sigma_{xg} +  b (\sigma_{g} - 1 )] & \sigma_{g} & - a_1^{-1} [ \sigma_{xg} +  b (\sigma_{g} - 1 )] & 0 \\ ~~\\ 
0 & 0 & 1 & 0 \\ ~~\\
- a_2^{-1} [ \sigma_x \, (\sigma_g)^{-1} +c ( (\sigma_g)^{-1} - 1)] & 0 & a_2^{-1} [ \sigma_x \, (\sigma_g)^{-1} +c ( (\sigma_g)^{-1} - 1)] 
& (\sigma_g)^{-1} \\
\end{pmatrix} \nn 
\end{align}
We conclude that as a crossed module the action on $\mbox{Char}^{ext}(T_{2})$ is not trivial.
\end{exa}

\appendix

\section{The classical gauge groupoid}\label{gaugegroupoid}

We recall in this section some basic facts about the gauge groupoid associated to a principal bundle and to the corresponding group of bisections; we shall mostly follow the book \cite{KirillMackenzie}. Let $\pi: P\to M$ be a principal bundle over the manifold $M$ with structure Lie group $G$. Consider the diagonal action of $G$ on $P \times P$ given by $(u, v)g:=(ug, vg)$; denote by $[u, v]$
the orbit of $(u, v)$ and by $\Omega= P\times_{G} P$ the collection of orbits. Then $\Omega$ is a groupoid over $M$, 
--- the \emph{gauge} or \emph{Ehresmann groupoid} of the principal bundle,  with source and target projections given by
\begin{align}
    s([u, v]):=\pi(v), \qquad t([u, v]):=\pi(u).
\end{align}
The object inclusion map is $M\to P\times_{G} P$ is given by
\begin{align}
m \mapsto \id_m :=[u, u]
\end{align}
for $m\in M$ and $u$ any element in $\pi^{-1}(m)$. And the partial multiplication $[u, v']\cdot [v, w]$, defined when $\pi(v')=\pi(v)$ is given by
\begin{align}
    [u, v]\cdot [v', w]=[u, w g],
\end{align}
for the unique $g\in G$ such that $v=v' g$.
\footnote{Here one is really using the classical translation map $t: P\times_{M} P\to G$,
$(ug ,u) \mapsto g$.}
One can always choose representatives such that $v=v'$ and the multiplication is then simply $[u, v]\cdot [v, w]=[u, w]$.
The inverse is
\beq\label{inv-grp}
    [u, v]^{-1} = [v, u] .
\eeq

A \textit{bisection} of the groupoid $\Omega$ is a map $\sigma: M\to \Omega$, which is right-inverse to the source projection,
$s\circ \sigma=\id_{M}$, and is such that $t\circ \sigma : M \to M$ is a diffeomorphism. The collection of bisections,
denoted $\B(\Omega)$, form a group: given two bisections $\sigma_{1}$ and $\sigma_{2}$, their multiplication is defined by
\begin{align}
    \sigma_{1}\ast\sigma_{2}(m):=\sigma_{1} \big((t\circ \sigma_{2})(m)\big)\sigma_{2}(m), \qquad \mbox{for} \quad m\in M .
\end{align}
The identity is the object inclusion $m \mapsto \id_m$, simply denoted $\id$, with inverse given by
\begin{align}
    \sigma^{-1}(m)= \Big(\sigma\big((t \circ\sigma)^{-1}(m)\big)\Big)^{-1};
\end{align}
here $(t \circ\sigma)^{-1}$ as a diffeomorphism of $M$, while the outer inversion is the one in \eqref{inv-grp}.

The subset $\B_{P/G}(\Omega)$ of `vertical' bisections, that is those bisections that
are right-inverse to the target projection as well, $t\circ \sigma=\id_{M}$, form a subgroup of $\B(\Omega)$.

It is a classical result \cite{KirillMackenzie} that there is a group isomorphism between
$\B(\Omega)$ and the group of principal ($G$-equivariant) bundle automorphisms of the principal bundle,
\begin{equation}
\mathrm{Aut}_{G}(P) :=\{\varphi:P\to P \, ; \,\, \varphi (pg) = \varphi(p)g  \} \, ,
\end{equation}
 while $\B_{P/G}(\Omega)$ is group isomorphic to
the group of gauge transformations, that is the subgroup of principal bundle
automorphisms which are vertical (project to the identity on the base space):
\begin{equation}\label{AVclass}
\mathrm{Aut}_{P/G}(P) :=\{\varphi:P\to P \, ; \,\, \varphi (pg) = \varphi(p)g ~, \, \, \pi( \varphi(p)) = \pi(p) \} .
\end{equation}

\vspace{.5cm}
	
\noindent
{\bf Acknowledgment:} 
We are most grateful to Chiara Pagani for many useful discussions. GL was partially supported by INFN, Iniziativa Specifica GAST and INDAM-GNSAGA.


\begin{thebibliography}{99}

\bibitem{ppca}
P. Aschieri, P. Bieliavsky, C. Pagani, A. Schenkel,
\emph{Noncommutative principal bundles through twist deformation}.
Commun. Math. Phys.  352 (2017) 287--344.

\bibitem{pgc}
P. Aschieri, G. Landi, C. Pagani,
\emph{The Gauge group of a noncommutative principal bundle and twist deformations}.
arXiv:1806.01841  [math.QA]; to appear in JNcG.

\bibitem{Bich} 
J. Bichon,
\emph{Hopf-Galois objects and cogroupoids}.
Rev. Un. Mat. Argentina, 5 (2014) 11--69.

\bibitem{brz-tr}
T. Brzezi\'nski,
\emph{Translation map in quantum principal bundles}.
J. Geom. Phys. 20 (1996) 349--370.

\bibitem{BW} 
T. Brzezi\'nski, R. Wisbauer, 
\emph{Corings and comodules}. 
London Mathematical Society Lecture Notes 309, CUP 2003.

\bibitem{Boehm} 
G. B\"ohm,
\emph{Hopf algebroids}. Handbook of Algebra, Vol. 6,
North-Holland, 2009,
Pages 173--235.

\bibitem{HM16} 
P.M. Hajac, T. Maszczyk,
\emph{Pullbacks and nontriviality of associated noncommutative vector bundles}.
arXiv:1601.00021 [math.KT].

\bibitem{kassel-review} 
C. Kassel,
\emph{Principal fiber bundles in non-commutative geometry}. In: Quantization, geometry and noncommutative structures in mathematics and physics, 
Mathemathical Physics Studies, Springer 2017, pp. 75--133.

\bibitem{Masuoka}
A. Masuoka, 
\emph{Cleft extensions for a Hopf algebra generated by a nearly primitive element}.
Comm. Algebra 22 (1994) 4537--4559.

\bibitem{KirillMackenzie}
K. Mackenzie,
\emph{General theory of Lie groupoids and Lie algebroids}. 
London Mathematical Society Lecture Notes Series 213, CUP 2005.

\bibitem{ralf-chen} 
R. Meyer, C. Zhu, 
\emph{Groupoids in categories with pretopology}. 
Theory Appl. Categ. 30 (2015) pp. 1906--1998.

\bibitem{mont} 
S. Montgomery,
\emph{Hopf algebras  and their actions on rings}.  AMS 1993.

\bibitem{schau}  
P. Schauenburg,
\emph{Hopf bi-Galois extensions}.
Commun. Algebra 24 (1996) 3797--3825.

\bibitem{schau1}  
P. Schauenburg,
\emph{Bi-Galois objects over the Taft algebras}.
Israel J. Math. 115 (2000) 101--123.

\bibitem{kornel}  
K. Szlachanyi,
\emph{Monoidal Morita equivalence}.
 	arXiv:math/0410407 [math.QA].

\bibitem{Taft}
E.J. Taft,
\emph{The order of the antipode of finite-dimensional Hopf algebra}
Proc. Nat. Acad. Sci. USA, 68 (1971) 2631--2633.

\bibitem{Ul89}
K.H. Ulbrich,  
\emph{Fiber functors of finite dimensional comodules}, 
Manuscripta Math. 65 (1989) 39--46.

\end{thebibliography}
\end{document}